\documentclass[a4paper,12pt]{article}

\usepackage{amsmath}
\usepackage{amsfonts}
\usepackage{amssymb}
\usepackage{amsthm}

\usepackage{graphicx}

 \usepackage{color}

\definecolor{blue}{rgb}{0,0,0.8}
\newcommand{\ts}{\textstyle}

\newcommand{\pt}{\partial}

\newcommand {\beq} {\begin{equation}}
\newcommand {\eeq} {\end{equation}}
\newcommand {\beqa} {\begin{eqnarray}}
\newcommand {\eeqa} {\end{eqnarray}}
\newcommand {\beqann} {\begin{eqnarray*}}
\newcommand {\eeqann} {\end{eqnarray*}}

\newcommand{\eps}{\varepsilon}
\newcommand{\R}{\mathbb{R}}
\newcommand{\norm}[2]{\|{#1}\|_{#2}}

\numberwithin{equation}{section}

\theoremstyle{plain}
\newtheorem{thm}{Theorem}[section]
\newtheorem{lem}[thm]{Lemma}

\newtheorem{cor}[thm]{Corollary}
\newtheorem{rem}[thm]{Remark}

\title{Green's function estimates for a 2d singularly perturbed
       convection-diffusion problem:\\{}\color{blue} extended analysis
}
\author{Sebastian~Franz\footnote{%
Institute of Scientific Computing, Technische Universit\"at Dresden, Germany;
         e-mail: sebastian.franz@tu-dresden.de}
        \and
        Natalia~Kopteva\footnote{Department of Mathematics and Statistics,
         University of Limerick,
         Limerick,
         Ireland;
         e-mail: natalia.kopteva@ul.ie}
        }

\date{}

\begin{document}

  \maketitle
   \begin{abstract}
   {\color{blue}This paper presents an extended version of the article \cite{FK10_1} by Franz and Kopteva.
   The main improvement compared to \cite{FK10_1} is in that here we additionally estimate the mixed second-order derivative of the Green's function.
   The case of Neumann conditions along the characteristic boundaries is also addressed.}

       A singularly perturbed convection-diffusion
    problem is posed in the unit square with a horizontal convective direction.
    Its solutions exhibit parabolic and exponential boundary layers.
    Sharp estimates of the Green's function and its
    first- and second-order derivatives are derived in the $L_1$ norm.
    The dependence of these estimates on the small diffusion parameter
    is shown explicitly.
    The obtained estimates will be used in a forthcoming numerical analysis
    of the considered problem.



   \end{abstract}

\section{Introduction}
   {\color{blue}This paper presents an extended version of the article \cite{FK10_1} by Franz and Kopteva.
   The main improvement compared to \cite{FK10_1} is in that here we additionally estimate the mixed second-order derivative of the Green's function.
   The case of Neumann conditions along the characteristic boundaries is also addressed;
   see Remarks~\ref{rem_int_G_Nbc} and~\ref{rem:def_tilde_bar_G}.
   (Most of the new material will be highlighted in the blue colour.)}

We investigate the Green's function for
the following problem posed in the unit-square domain
$\Omega=(0,1)^2$:
\begin{subequations}\label{eq:Lu}
\begin{align}
\label{eq:Lu_a}
   L_{xy}u(x,y):=-\eps(u_{xx}+u_{yy})-(a(x,y)\,u)_x+b(x,y)\,u&=f(x,y)\quad
   \mbox{for }(x,y)\in\Omega,\\
   u(x,y)&=0\qquad\quad\;\,\mbox{for }(x,y)\in\partial\Omega.
\end{align}
\end{subequations}
Here $\eps$ is a small positive parameter, while the coefficients
$a$ and $b$  are sufficiently smooth (e.g., $a,\,b\in
C^\infty(\bar\Omega)$).
We also assume,  for some positive constant $\alpha$, that
\beq\label{assmns}
   a(x,y)\geq \alpha>0,\quad b(x,y)\ge 0,\quad b(x,y)-
   a_x(x,y)\geq 0
   \qquad\mbox{for~all~}(x,y)\in\bar\Omega.
\eeq Under these assumptions, \eqref{eq:Lu_a} is a singularly
perturbed elliptic equation, frequently referred to as a
convection-dominated convection-diffusion equation. This equation
serves as a model for Navier-Stokes equations at large Reynolds
numbers or (in the linearised case) of Oseen equations and
provides an excellent paradigm for numerical techniques in the
computational fluid dynamics \cite{RST08}.

The asymptotic analysis for problems of type \eqref{eq:Lu} is
very intricate and illustrates the complexity of their solutions
\cite[Section IV.1]{Ilin}, \cite{KSh87}.
We also refer the reader to \cite[Chapter~IV]{Shi92} and
\cite{KSt05,KSt07} for pointwise estimates of solution
derivatives.
In short, solutions of problem \eqref{eq:Lu} typically exhibit
parabolic boundary layers along the characteristic boundaries
$y=0$ and $y=1$, and an exponential boundary layer along the
outflow boundary $x=0$. Furthermore, if a discontinuous Dirichlet
boundary condition is imposed at the inflow boundary $x=1$, then
solutions also exhibit characteristic interior layers.
Note that because of the complexity of the solutions, the analysis
techniques \cite{KSt05,KSt07} work only for a constant-coefficient
version of \eqref{eq:Lu_a}.
Note also that the complex solution structure 
is reflected in the corresponding Green's function, which is the
subject of this paper.

Our interest in considering the Green's function of problem
\eqref{eq:Lu} and estimating its derivatives is motivated by the
numerical analysis of this computationally challenging problem.
More specifically, we shall use the obtained estimates in the
forthcoming paper \cite{FK10_NA} to derive robust a posteriori
error bounds for computed solutions of this problem  using
finite-difference methods. (This approach is related to recent
articles \cite{Kopt08,CK09}, which address the numerical solution
of singularly perturbed equations of reaction-diffusion type.)
In a more general numerical-analysis context, 
we note that sharp estimates for continuous Green's functions (or
their generalised versions) frequently play a crucial role in a
priori and a posteriori error analyses \cite{erikss,Leyk,notch}.

We shall estimate the derivatives of the Green's function  in the
$L_1$ norm (as they will be used to estimate the error in the
computed solution in the dual $L_\infty$ norm \cite{FK10_NA}).
Our estimates will be  \emph{uniform in the small perturbation
parameter $\eps$} in the sense that any dependence on $\eps$ will
be shown explicitly.
Note also that our estimates will be {\it sharp}
(in the sense of Theorem~\ref{theo_lower})
up to an $\eps$-independent constant multiplier.

As any Green's function estimate implies a certain a priori
estimate for the original problem, we also refer the reader to
 D\"orfler \cite{Dorf99}, who, for a similar problem,
gives extensive a priori solution estimates that involve the
right-hand side in various positive norms such as $L_p$ and
$W^{m,p}$ with $m\ge 0$.
In comparison, a priori solution estimates that
follow from our results, involve negative norms of the right-hand
side (see Corollary~\ref{cor_apriori} and also
Remark~\ref{rem_apriori}), so they are different in nature.

Our analysis in this paper resembles 
those in \cite[Section~3]{Kopt08}, \cite[Section~3]{CK09} in that,
roughly speaking, we freeze the coefficients and estimate the
corresponding explicit Green's function for a constant-coefficient
equation, and then we investigate the difference between the
original and the frozen-coefficient Green's functions.
The two cited papers deal with equations of reaction-diffusion
type, for which the Green's function in the unbounded domain is
(almost) radially symmetric and exponentially decaying away from the singular point.
By contrast, the Green's function for the convection-diffusion
problem \eqref{eq:Lu} exhibits
a much more complex 
{\it anisotropic} structure (see Fig.~\ref{fig:Green}).
   \begin{figure}[tb]
      \centerline{
     \includegraphics[width=0.5\textwidth]{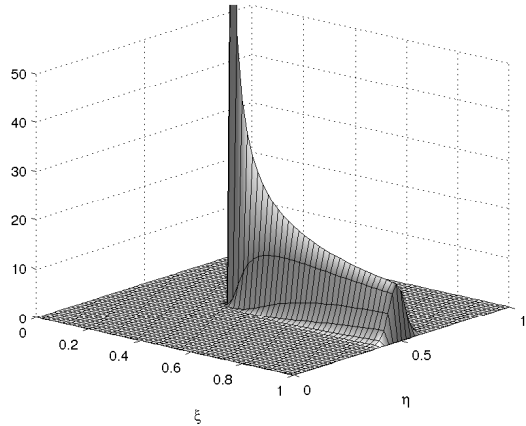}   }
      \caption{Typical anisotropic behaviour of the Green's function for problem \eqref{eq:Lu}:
               $a=1$, $b=0$, $(x,y)=(\frac13,\frac12)$ and $\eps=10^{-3}$.}
             \label{fig:Green}
   \end{figure}
This is reflected in a much more intricate analysis compared to
\cite{Kopt08,CK09}, in particular, for the variable-coefficient
case.

The paper is organised as follows.
In Section~\ref{sec:def}, the Green's function 
associated with problem \eqref{eq:Lu} is defined and
upper bounds for its derivatives are stated
in Theorem~\ref{thm:main},
which is
the main result of the paper.
The corresponding lower bounds are then given in Theorem~\ref{theo_lower}.
In Section~\ref{sec:def_green_const}, we obtain the fundamental solution
for a constant-coefficient version of \eqref{eq:Lu_a}
in the domain $\Omega=\R^2$;
this fundamental solution is bounded in
Section~\ref{sec:bounds_green_general}.
Next, in Section~\ref{sec:bounds_green_const},
using the method of images with an inclusion of cut-off
functions, we
define and estimate certain approximations of the
constant-coefficient Green's functions
in the domains $\Omega=(0,1)\times\R$ and $\Omega=(0,1)^2$.
The difference between the constant-coefficient approximations of Section~\ref{sec:bounds_green_const}
and the original variable-coefficient Green's function
is estimated in Section~\ref{sec:main_proof}; this completes the proof of
Theorem~\ref{thm:main}.
In the final Section~\ref{sec:general} we discuss generalisation of our results to more than two
dimensions.
\smallskip

\textit{Notation.} Throughout the paper, $C$ denotes a generic positive constant
that may take different values in different formulas, but is
{\it independent of the
singular perturbation parameter $\eps$}.
A subscripted $C$ (e.g., $C_1$) denotes a positive constant
that takes a fixed value, and is also independent of $\eps$.
Notation such as
$v = \mathcal{O}(w)$ means $|v|\le Cw$ for some $C$.
The standard Sobolev spaces $W^{m,p}(\Omega')$ and $L_p(\Omega')$
on any measurable
subset $\Omega'\subset\R^2$ are used for $p\geq 1$ and $m=1,2$.
The $L_p(\Omega')$ norm is denoted by $\norm{\cdot}{p\,;\Omega'}$
while the $W^{m,p}(\Omega')$ norm is denoted by $\norm{\cdot}{m,p\,;\Omega'}$.
Sometimes the domain of interest will be an open ball
$B(x',y';\rho):=\{(x,y)\in\R^2: (x-x')^2+(y-y')^2<\rho^2\}$
centred at $(x',y')$ of radius~$\rho$.
For the partial derivative of a function $v$ in a variable $\xi$ we will use the
equivalent notations $v_\xi$ and $\pt_\xi v$.
Similarly,
$v_{\xi\xi}$ and $\pt^2_{\xi}v$ both denote
the second-order pure
derivative of $v$ in $\xi$, while
$v_{\xi\eta}$ and $\pt^2_{\xi\eta}v$ both denote
the second-order mixed
derivative of $v$ in $\xi$ and~$\eta$.

\section{Definition of the Green's function. Main result}\label{sec:def}
%
   Let $G=G(x,y;\xi,\eta)$ be the Green's function
   associated with problem \eqref{eq:Lu}. For each fixed $(x,y)\in\Omega$, it satisfies
   \beq
   \label{eq:Green_adj}
   \begin{array}{rll}
     \!\! L^*_{\xi\eta}G(x,y;\xi,\eta)
          =-\eps(G_{\xi\xi}+G_{\eta\eta})+a(\xi,\eta)G_\xi+b(\xi,\eta)G
         \!\!\!\!\!&{}=\delta(x-\xi)\,\delta(y-\eta),
     \!\!\! &(\xi,\eta)\in\Omega,\\
     G(x,y;\xi,\eta)\!\!\!\!\!&{}=0,&(\xi,\eta)\in\partial\Omega.
     \end{array}
   \eeq
   Here $L^*_{\xi\eta}$ is the adjoint differential operator to $L_{xy}$,
   while $\delta(\cdot)$ is the one-dimensional Dirac $\delta$-distribution,
   so the product $\delta(x-\xi)\,\delta(y-\eta)$ is equivalent
   to the two-dimensional $\delta$-distribution centred at
   $(\xi,\eta)=(x,y)$; see \cite[Example 3.29]{Griffel}, \cite[Section 5.5]{Stakgold2}.
   The unique solution $u$ of \eqref{eq:Lu} has the representation
   \begin{gather}\label{eq:sol_prim}
     u(x,y)=\iint_{\Omega}G(x,y;\xi,\eta)\,f(\xi,\eta)\,d\xi\, d\eta,
   \end{gather}
   (provided that $f$ is sufficiently regular so that
   \eqref{eq:sol_prim} is well-defined).
   Note that, for each fixed $(\xi,\eta)\in\Omega$,
   the Green's function $G$ also satisfies
   \beq
   \label{eq:Green_prim}
   \begin{array}{rll}
     \!\!\! L_{xy}G(x,y;\xi,\eta)
         =-\eps(G_{xx}+G_{yy})\!-\!(a(x,y)G)_x+b(x,y)G
         \!\!\!\!\!&{}=\delta(x-\xi)\delta(y-\eta),
     \!\!\!\! &(x,y)\!\in\!\Omega,\\
     G(x,y;\xi,\eta)\!\!\!\!\!&{}=0,&(x,y)\!\in\!\partial\Omega.
     \end{array}
   \eeq
   Therefore, the unique solution $v$ of the adjoint problem
%
   \begin{align*}
      L^*_{xy}v(x,y)=-\eps(v_{xx}+v_{yy})+a(x,y)\,v_x+b(x,y)\,v&=f(x,y)\quad
     \mbox{for }(x,y)\in\Omega,\\
     v(x,y)&=0\qquad\quad\;\,\mbox{for }(x,y)\in\partial\Omega.
   \end{align*}
%
   is given by
   \vspace{-0.3cm}
   \begin{gather}\label{eq:sol_adj}
     v(\xi,\eta)=\iint_{\Omega}G(x,y;\xi,\eta)\,f(x,y)\,dx\, dy.
   \end{gather}

   We first give a preliminary result for $G$.

   \begin{lem}\label{lem_G_L1}
     The Green's function $G$ 
     associated with problem
     \eqref{eq:Lu} satisfies
     \begin{gather}\label{G_L1}
       \int_0^1 |G(x,y;\xi,\eta)|\,d\eta       \leq C
       ,\qquad
       \|G(x,y;\cdot)\|_{1\,;\Omega}           \leq C
       \qquad\mbox{for~}(x,y)\in\Omega,
     \end{gather}
     where $C$ is some positive $\eps$-independent constant.
   \end{lem}

   \begin{proof}
      The first estimate of \eqref{G_L1} is given in the proof of
      \cite[Theorem 2.10]{Dorf99} (see also \cite[Theorem~III.1.22]{RST08} and
      \cite{And_2003} for similar results).
      The second desired estimate  follows.
   \end{proof}

   \begin{rem}\label{rem_int_G_Nbc}\color{blue}
   Note that the result of Lemma~\ref{lem_G_L1} remains true for the case of homogeneous Neumann boundary conditions
   in \eqref{eq:Green_adj} at the top and bottom boundaries (while Dirichlet boundary conditions remain unchanged
   at the right and left boundaries).
   To prove this, fix any $y\in(0,1)$ and let $I_G(x;\xi):=\int_0^1 G(x,y;\xi,\eta)\,d\eta $.
   Then a calculation using the Neumann boundary conditions, shows that
   $-\eps\pt_\xi^2 I_G+\pt_\xi[\bar a(x;\xi) I_G]+\bar b(x;\xi)I_G=\delta(x-\xi)$ subject to $I_G=0$ at $\xi=0,1$.
   Here,
   $\bar a :=I_G^{-1}\int_0^1 a(\xi,\eta) G(x,y;\xi,\eta)\,d\eta $, and similarly
   $\bar b :=I_G^{-1}\int_0^1 [b-\pt_\xi a](\xi,\eta) G(x,y;\xi,\eta)\,d\eta $, for which, in view of $G\ge 0$ and
    \eqref{assmns}, one has $\bar a\ge\alpha>0$ and $\bar b\ge 0$. Now,
   \cite[Theorem~2.1]{And_MM02} yields $I_G\le \alpha^{-1}$.
   \end{rem}

   We now state the main result of this paper.
   \begin{thm}\label{thm:main}
      Let $\eps\in(0,1]$.
      The Green's function  $G$ 
      associated with \eqref{eq:Lu}
      on the unit square $\Omega=(0,1)^2$  satisfies, for all $(x,y)\in\Omega$,
      the following bounds
      \begin{subequations}
      \begin{align}
\label{eq:thm:G_xi}
         \norm{\pt_\xi  G(x,y;\cdot)}{1\,;\Omega}
            &\leq C(1+|\ln \eps|),\\
\label{eq:thm:G_eta}
         \norm{\pt_\eta G(x,y;\cdot)}{1\,;\Omega}+\norm{\pt_y G(x,y;\cdot)}{1\,;\Omega}
            &\leq C\eps^{-1/2}.
      \end{align}
      Furthermore, for any ball $B(x',y';\rho)$ of radius $\rho$
      centred at any
     $(x',y')\in\bar\Omega$, we have
      \begin{align}
         \norm{G(x,y;\cdot)}{1,1\,;B(x',y';\rho)}
            &\leq C\eps^{-1}\rho,\label{eq:thm:G_grad}
      \end{align}
      while for the ball $B(x,y;\rho)$ of radius $\rho$ centred at $(x,y)$
      we have
      \begin{align}
         \norm{\pt^2_{\xi}  G(x,y;\cdot)}{1\,;\Omega\setminus B(x,y;\rho)}
            &\leq C\eps^{-1}\ln(2+\eps/\rho),\label{eq:thm:G_xixi}\\
                     \color{blue}\norm{\pt^2_{\xi\eta}  G(x,y;\cdot)}{1\,;\Omega\setminus B(x,y;\rho)}
            &\color{blue}\leq C\eps^{-1}\ln(2+\eps/\rho),\label{eq:thm:G_xieta}\\
         \norm{\pt^2_{\eta}G(x,y;\cdot)}{1\,;\Omega\setminus B(x,y;\rho)}
            &\leq C\eps^{-1}(\ln(2+\eps/\rho)+|\ln\eps|).\label{eq:thm:G_etaeta}
      \end{align}
      \end{subequations}
      Here $C$ is some positive $\eps$-independent constant.
   \end{thm}
The rest of the paper is devoted to the proof of this theorem, which is completed in
Section~\ref{sec:main_proof}.

In view of the solution representation \eqref{eq:sol_prim}, the bounds \eqref{eq:thm:G_xi}, \eqref{eq:thm:G_eta}
immediately imply the
following a priori solution estimates for our original problem.
\begin{cor}\label{cor_apriori}
Let $f(x,y)=\pt_x F_1(x,y)+\pt_y F_2(x,y)$ with $F_1,\,F_2\in L_\infty(\Omega)$.
Then for the solution $u$ of problem \eqref{eq:Lu} we have the bound
\beq\label{apriori}
\|u\|_{\infty\,;\Omega}
\le C\bigl[\,
(1+|\ln\eps|)\,\|F_1\|_{\infty\,;\Omega}+ \eps^{-1/2}\,\|F_2\|_{\infty\,;\Omega}
\,\bigr].
\eeq
\end{cor}
\begin{proof}
Represent $u$ using \eqref{eq:sol_prim}.
   Then integrate by parts and use \eqref{eq:thm:G_xi} and
   \eqref{eq:thm:G_eta}.
\end{proof}

\begin{rem}\label{rem_apriori}
Let us associate the components $\pt_x F_1$ and $\pt_yF_2$ of $f$
with the one-dimensional parts $-\eps\pt_x^2-\pt_x a(x,y)$
and $-\eps\pt_y^2+ b(x,y)$, respectively, of the operator $L_{xy}$.
Then, bar the weak logarithmic factor $|\ln\eps|$,
the bound \eqref{apriori} clearly resembles the corresponding one-dimensional
a priori solution estimates.
Indeed, for the one-dimensional equations
$-\eps u_1''(x)-(a_1(x)u_1(x))'=f_1(x)$
and
$-\eps u_2''(x)+b_2(x)u_2(x)=f_2(x)$ (where $a_1,b_2\ge C>0$)
subject to $u_{1,2}(0)=u_{1,2}(1)=0$, one has
$\|u_1\|_{\infty\,;(0,1)}\le C \|f_1\|_{-1,\infty\,;(0,1)}$,
and $\|u_2\|_{\infty\,;(0,1)}\le C \eps^{-1/2}\|f_2\|_{-1,\infty\,;(0,1)}$,
where
$\|\cdot\|_{-1,\infty\,;(0,1)}$ is the norm in the negative Sobolev space $W^{-1,
\infty}(0,1)$
(see, e.g., \cite[Theorem~3.25]{Linss}).

\end{rem}

Note that the upper estimates of Theorem~\ref{thm:main}
are {\it sharp} in the following sense.

\begin{thm}[\cite{FK09_lower}]\label{theo_lower}
Let $\eps\in(0,c_0]$ for some sufficiently small positive $c_0$.
Set $a(x,y):=\alpha$ and $b(x,y):=0$ in \eqref{eq:Lu}.
Then the Green's function $G$
associated with this problem
      on the unit square $\Omega=(0,1)^2$
 satisfies,
      for all $(x,y)\in[\frac14,\frac34]^2$,
      the following lower
      bounds:%
      \begin{subequations}\label{eq_thm:main_lower}
      \begin{align}
         \norm{\pt_{\xi} G(x,y;\cdot)}{1;\Omega}
            &\geq c\,|\ln \eps|,\\
         \norm{\pt_{\eta} G(x,y;\cdot)}{1;\Omega}
            &\geq c\,\eps^{-1/2}.
            \intertext{Furthermore, for any ball $B(x,y;\rho)$ of radius $\rho\le\frac18$, we have}
         \norm{G(x,y;\cdot)}{1,1;\Omega\cap B(x,y;\rho)}
            &\geq \begin{cases}
                                 c\,\rho/\eps, & \mbox{if~}\rho\le 2\eps,\\
                                          c\,(\rho/\eps)^{1/2},&\mbox{otherwise},\\
                                       \end{cases}\\
         \norm{\pt^2_{\xi_1} G(x,y;\cdot)}{1;\Omega\setminus B(x,y;\rho)}
                       &\geq c\,\eps^{-1}\ln(2+\eps/\rho),
           \quad&&\mbox{if~}\rho\le c_1\eps,
            \\
            \norm{\pt^2_{\xi_k} G(x,y;\cdot)}{1;\Omega\setminus B(x,y;\rho)}
            &\geq c\,\eps^{-1}(\ln(2+\eps/\rho)+|\ln\eps|),
           &&\mbox{if~}\rho\le{\textstyle\frac18}.
      \end{align}
     Here $c$ and $c_1$ are $\eps$-independent positive constants.
      \end{subequations}
\end{thm}
This result can be anticipated from an
inspection of the bounds for an explicit fundamental solution in a
constant-coefficient case; see Section~\ref{sec:bounds_green_general}.

   \section{Fundamental solution in a constant-coefficient case}\label{sec:def_green_const}
%
In this section we shall explicitly solve simplifications
of the two problems \eqref{eq:Green_adj} and \eqref{eq:Green_prim}
that we have for $G$. To get these simplifications,
 we freeze the coefficients in these problems by
replacing $a(\xi,\eta)$ by $a(x,y)$ in \eqref{eq:Green_adj}, and
replacing $a(x,y)$ by $a(\xi,\eta)$ in \eqref{eq:Green_prim},
and also setting $b:=0$; the frozen-coefficient versions of
the operators $ L^*_{\xi\eta}$ and $ L_{xy}$
will be denoted by $ \bar L^*_{\xi\eta}$ and $\tilde L_{xy}$, respectively.
Furthermore, we extend the resulting
equations to $\R^2$ and denote their solutions by $\bar g$ and $\tilde g$.
Thus we get
   \begin{align}\label{eq:Green_adj_const}
     \bar L^*_{\xi\eta}\bar g(x,y;\xi,\eta)
         =-\eps(\bar g_{\xi\xi}+\bar g_{\eta\eta})+a(x,y)\,\bar g_\xi
        &=\delta(x-\xi)\,\delta(y-\eta)
     \quad\mbox{for~}(\xi,\eta)\in\R^2,
\\
\label{eq:Green_prim_const}
     \tilde L_{xy}\tilde g(x,y;\xi,\eta)
          =-\eps(\tilde g_{xx}+\tilde g_{yy})-a(\xi,\eta)\,\tilde g_x
         &=\delta(x-\xi)\,\delta(y-\eta)\quad \mbox{for~}(x,y)\in\R^2.
   \end{align}
As the variables $(x,y)$ appear as parameters in equation \eqref{eq:Green_adj_const}
and $(\xi,\eta)$ appear as parameters in equation \eqref{eq:Green_prim_const},
we effectively have two equations with constant coefficients.
A calculation (see Remark~\ref{rem_const_coef} below for details) yields explicit representations of their solutions by
\beq\label{bar_tilde_g_def}
\bar g(x,y;\xi,\eta)=g(x,y;\xi,\eta;q)\Bigr|_{q={\textstyle\frac{1}{2}}a(x,y)},
\qquad
\tilde g(x,y;\xi,\eta)=g(x,y;\xi,\eta;q)\Bigr|_{q={\textstyle\frac{1}{2}}a(\xi,\eta)}.
\eeq
Here the function $g$ is defined,
using the modified Bessel function of the second kind of order zero
$K_0(\cdot)$,
 by
\begin{subequations}\label{eq:def_g0}
   \begin{align}
     &g =g(x,y;\xi,\eta;q)
       :={\frac{1}{2\pi\eps}}\, e^{q\hat\xi_{[x]}}\, K_0(q\hat{r}_{[x]}),\\
      \label{hat_xi_r}
      &\hat\xi_{[x]}:=(\xi-x)/\eps, \quad
       \hat{\eta}   :=(\eta-y)/\eps,\quad
       \hat{r}_{[x]}:=\sqrt{\hat\xi^2_{[x]}+\hat{\eta}^2}.
   \end{align}
\end{subequations}
We use a subindex in
$\hat\xi_{[x]}$ and $\hat r_{[x]}$
to highlight their dependence  on $x$
as in many places $x$ will take different values;
but when there is no ambiguity, we shall sometimes 
simply write
$\hat\xi$ and $\hat r$.

The function $g$ and its derivatives involve
the modified Bessel functions of the second kind of order zero
$K_0(\cdot)$  and of order one $K_1(\cdot)$.
   With the notation $K_{0,1}:=\max\{K_0,K_1\}$, we quote some useful properties
   of the modified Bessel functions~\cite{AS64}:
   \begin{subequations}
    \begin{align}
      K_{0,1}(s) & \leq C s^{-1}e^{-s/2}\quad \forall s>0, \qquad
      K_{0,1}(s)   \leq C s^{-1/2}e^{-s}\quad \forall s\geq C>0,\label{K_01_bound}\\
      K_0(z) & = K_1(z)\bigl[1-\frac1{2z}+\mathcal{O}(z^{-2})\bigr].\label{K_0_via_K_1}
    \end{align}
   \end{subequations}

\begin{rem}\label{rem_const_coef}
   The representation \eqref{eq:def_g0} is given in \cite[(III.1.16)]{RST08}.
   For completeness, we sketch a proof of \eqref{bar_tilde_g_def},\,\eqref{eq:def_g0}
   for $\bar g$.
   Set $q=\frac{1}{2}a(x,y)$ and $\bar g = V(\xi,\eta)\,e^{q\xi/\eps}$
   in \eqref{eq:Green_adj_const}.
   Now a calculation shows that
   \[
      -\eps^2(V_{\xi\xi}+V_{\eta\eta})+q^2 V
      =\eps \,e^{-q\xi/\eps}\,\delta(x-\xi)\,\delta(y-\eta).
   \]
   As the fundamental solution for the operator
   $-\eps^2(\pt^2_{\xi}+\pt^2_{\eta})+q^2$
   is $\frac{1}{2\pi\eps^2}K_0(q r/\eps)$ \cite{Kopt08}, the desired
   representation \eqref{bar_tilde_g_def},\,\eqref{eq:def_g0} for $\bar g$ follows.
\end{rem}

%
\section{Bounds for the fundamental solution $g(x,y;\xi,\eta;q)$}\label{sec:bounds_green_general}
%
   Throughout this section we assume that $\Omega=(0,1)\times\R$,
   but all results remain valid for $\Omega=(0,1)^2$.
   Here we derive a number of useful bounds for the fundamental solution $g$
   of \eqref{eq:def_g0}
   and its derivatives that will be used in Section~\ref{sec:bounds_green_const}.
   As sometimes $q=\frac12 a(x,y)$ or $q=\frac12 a(\xi,\eta)$
   (as in \eqref{bar_tilde_g_def}),
   we shall also use the full-derivative notation
   \beq\label{D_ops}
      D_\eta:=\pt_\eta+{\ts\frac12}\pt_\eta a(\xi,\eta)\cdot\pt_q,
      \qquad\quad
      D_y:=\pt_y+{\ts\frac12}\pt_y a(x,y)\cdot\pt_q.
   \eeq

   \begin{lem}\label{lem:g0_bounds}
    Let $(x,y)\in[-1,1]\times\R$ and $0<\frac12\alpha\le q\le C$.
    Then for the function $g=g(x,y;\xi,\eta;q)$
    of \eqref{eq:def_g0}
    we have the following bounds
    \begin{subequations}\label{g_bounds}
     \begin{align}
        \| g(x,y;\cdot;q)\|_{1\,;\Omega}
         &\leq C,\label{g_L1}\\
        \|\pt_\xi g(x,y;\cdot;q)\|_{1\,;\Omega}
         &\leq C(1+|\ln\eps|),\label{g_xi_L1}\\
        \eps^{1/2}\,\|\pt_\eta g(x,y;\cdot;q)\|_{1\,;\Omega}+
                    \|\pt_q g(x,y;\cdot;q)\|_{1\,;\Omega}
         &\leq C,\label{g_eta_L1}\\
        \|(\eps \hat r_{[x]}\,\pt_{\xi} g)(x,y;\cdot;q)\|_{1\,;\Omega}
         &\leq C,\label{g_xi_R_L1}\\
        \eps^{1/2}\,\|(\eps \hat r_{[x]}\,\pt^2_{\xi\eta} g)(x,y;\cdot;q)\|_{1\,;\Omega}+
                    \|(\eps \hat r_{[x]}\,\pt^2_{\xi q} g)(x,y;\cdot;q)\|_{1\,;\Omega}
         &\leq C,\label{g_eta_x_L1}
     \end{align}
     and for any ball $B(x',y';\rho)$ of radius $\rho$ centred at any
     $(x',y')\in[0,1]\times\R$, we have
     \begin{align}
        \|g(x,y;\cdot;q)\|_{1,1\,; B(x',y';\rho)}
         &\leq C\eps^{-1}\rho,\label{g_eta_ball}
     \intertext{while for the ball $B(x,y;\rho)$ of radius $\rho$ centred at $(x,y)$, we have}
     \hspace{-0.3cm}
        \|\pt^2_{\xi} g(x,y;\cdot;q)\|_{1\,;\Omega\setminus B(x,y;\rho)}
         &\leq C\eps^{-1}\ln(2+\eps/\rho),\label{g_xi2_L1}\\
               \color{blue} \|\pt^2_{\xi\eta} g(x,y;\cdot;q)\|_{1\,;\Omega\setminus B(x,y;\rho)}
         &\color{blue}\leq C\eps^{-1}\ln(2+\eps/\rho),\label{g_xieta2_L1}\\
        \|\pt^2_{\eta} g(x,y;\cdot;q)\|_{1\,;\Omega\setminus B(x,y;\rho)}
         &\leq C\eps^{-1}(\ln(2+\eps/\rho)+|\ln\eps|).\label{g_eta2_L1}
     \end{align}
    \end{subequations}
   Furthermore, one has the bound
    \begin{subequations}
      \begin{gather}
        \|\pt_x g(x,y;\cdot;q)\|_{1\,;\Omega}
         \leq C(1+|\ln\eps|),\label{g_x_L1}
      \end{gather}
     and, with the full-derivative notation \eqref{D_ops}, the bounds
      \begin{align}
        \|D_\eta g(x,y;\cdot;q)\|_{1\,;\Omega}+\|D_y g(x,y;\cdot;q)\|_{1\,;\Omega}
         &\leq C\eps^{-1/2},\label{D_g}\\
        \|(\eps \hat r_{[x]}\,D_\eta\pt_x g)(x,y;\cdot;q)\|_{1\,;\Omega}+
        \|(\eps \hat r_{[x]}\,D_y\pt_\xi g)(x,y;\cdot;q)\|_{1\,;\Omega}
         &\leq C\eps^{-1/2}.\label{D_pt_g}
      \end{align}
    \end{subequations}
   \end{lem}

\begin{proof}
First, note that $\pt_x g=-\pt_\xi g$ and $\pt_y g=-\pt_\eta g$, so
\eqref{g_x_L1} follows from \eqref{g_xi_L1},
\eqref{D_g} follows from \eqref{D_ops},\,\eqref{g_eta_L1},
while \eqref{D_pt_g} follows from \eqref{D_ops},\,\eqref{g_eta_x_L1}.
Thus it suffices to establish the bounds \eqref{g_bounds}.

 Throughout the proof, $x$ and $y$ are fixed so we employ the notation
$\hat\xi:=\hat\xi_{[x]}$ and $\hat r:=\hat r_{[x]}$.
A calculation shows that the first-order derivatives of $g(x,y;\xi,\eta;q)$
are given by
\begin{subequations}
\begin{align}
   \label{g_xi}
   \partial_\xi g
      &= {\ts \frac{q}{2\pi\eps^2}}\,e^{q\hat\xi}\,
         \Bigl[ K_0(q\hat{r})-\frac{\hat\xi}{\hat{r}}K_1(q\hat{r})\Bigr],\\
   \label{g_eta}
   \partial_\eta g
      &=-{\ts \frac{q}{2\pi\eps^2}}\,e^{q\hat\xi}\,
         \Bigl[\frac{\hat\eta}{\hat{r}}K_1(q\hat{r})\Bigr],\\
   \label{g_q}
   \partial_q g
      &= {\ts \frac{1}{2\pi\eps}}\,\hat r e^{q\hat\xi}\,
         \Bigl[\frac{\hat\xi}{\hat{r}} K_0(q\hat{r})- K_1(q\hat{r})\Bigr].
\end{align}
\end{subequations}
Here we used $K'_0=-K_1$, \cite{AS64}, and then
 $\partial_\xi\hat r=\eps^{-1}\hat\xi/\hat r$
and $\partial_\eta\hat r=\eps^{-1}\hat\eta/\hat r$.
In a similar manner, but additionally using
$K_1'(s)=-K_0(s)-K_1(s)/s$ \cite{AS64}, and also
$\pt_\xi(\hat\eta/\hat r)=-\eps^{-1}\hat\xi\hat\eta/\hat r^3$
and $\pt_\eta(\hat\eta/\hat r)=\eps^{-1}\hat\xi^2/\hat r^3$,
one gets the second-order derivatives
\begin{subequations}
\begin{align}
   \label{g2_xi_eta}
   \pt^2_{\xi\eta} g
    &= {\ts\frac{q}{2\pi\eps^3}}\,e^{q\hat\xi}\,\frac{\hat\eta}{\hat r^2}
       \Bigl[q {\hat r}
         \Bigl( \frac{\hat\xi}{\hat r} K_0(q\hat r)
             -K_1(q\hat r)\Bigr)
         +2\frac{\hat\xi}{\hat r}K_1(q\hat r)
       \Bigr],\\
   \label{g2_xi_q}
   \partial^2_{\xi q} g
    &= {\ts \frac{q}{2\pi\eps^2}}\,e^{q\hat\xi}\,\hat r^{-1}
       \Bigl[\hat\xi\hat r \{2K_0(q\hat{r})+{\ts\frac1{ q\hat r}}K_1(q\hat{r})\}
             -(\hat\xi^2+\hat{r}^2)K_1(q\hat{r})\Bigr]
      +q^{-1}\partial_{\xi} g,\\
   \label{g2_eta_eta}
   \pt^2_{\eta} g
   &= {\ts\frac{q}{2\pi\eps^3}}\, e^{q\hat\xi}
      \Bigl[  q \,\frac{{\hat\eta}^2}{{\hat r}^2}K_0(q\hat r)
            +\frac{\hat\eta^2-\hat\xi^2}{\hat r^3} K_1(q\hat r)\Bigr].
\intertext{Finally, combining $\pt_\xi^2 g=-\pt_\xi^2 g+\frac{2q}\eps\,\pt_\xi g$
with (\ref{g_xi}) and (\ref{g2_eta_eta}) yields}
   \label{g2_xi_xi}
   \pt^2_{\xi} g
    &= {\ts\frac{q}{2\pi\eps^3}}\,e^{q\hat\xi}\,
       \Bigl[q\Bigl( K_0(q\hat r)
                    +\frac{\hat\xi^2}{\hat r^2}K_0(q\hat r)
                    -2\frac{\hat\xi}{\hat r}K_1(q\hat r)
              \Bigr)
         +\frac{\hat\xi^2-\hat\eta^2}{\hat r^3}K_1(q\hat r)
       \Bigr].
\end{align}
\end{subequations}

Now we proceed to estimating the above  derivatives of $g$.
Note that $d\xi\,d\eta=\eps^2 d\hat\xi\,d\hat\eta$,
where $(\hat\xi,\hat\eta)\in\hat\Omega:=\eps^{-1}(-x,1-x)\times\R
\subset(-\infty,2/\eps)\times\R$.
Consider the domains
\[
   \hat\Omega_1:=\bigl\{\hat\xi<1+{\ts\frac14|\hat\eta|}\bigr\},
   \qquad
   \hat\Omega_2:=\bigl\{\max\{\ts1,\frac14|\hat\eta|\}<\hat\xi<2/\eps\bigr\}.
\]
As  $\hat\Omega\subset\hat\Omega_1\cup\hat\Omega_2$ for any $x\in[-1,1]$,
it is convenient to consider integrals over these two subdomains separately.

(i) Consider $(\hat\xi,\hat\eta)\in\hat\Omega_1$.
Then $\hat\xi\le1+\frac14\hat r$ so,
with the notation $K_{0,1}:=\max\{K_0,K_1\}$,
one gets
\begin{align}
   \eps^2\bigl[(1+\hat r) (\eps^{-1}|g|+|\pt_\xi g|+|\pt_\eta g|+|\pt_q g|
                           +|\pt^2_{\xi q} g|)
               +\eps\hat r|\pt^2_{\xi\eta} g|\bigr]
    &\leq Ce^{q\hat\xi} (1+\hat r+\hat r^2) K_{0,1}(q\hat r)\notag\\
    &\leq C\hat r^{-1}  e^{-q\hat r/8}, \label{star0}
\end{align}
where we combined $e^{q\hat\xi}\le e^{q(1+\hat r/4)}$
with $1+\hat r+\hat r^2\le Ce^{q\hat r/8}$
(which follows from $q\ge\frac12\alpha$)
and $K_{0,1}(q\hat r) \leq C (q\hat r)^{-1}e^{-q\hat r/2}$
(see \eqref{K_01_bound}).
This immediately yields
\begin{align}
   \iint_{\hat\Omega_1}\!\bigl[(1+\hat r)(\eps^{-1}|g|+|\pt_\xi g|+|\pt_\eta g|+|\pt_q g|
                                          +|\pt^2_{\xi q} g|)
                               +\eps\hat r|\pt^2_{\xi\eta} g|\bigr]
                       \,\bigl(\eps^2d\hat\xi \,d\hat\eta\bigr)
   &\leq C\!\int_0^{\infty}\!\!\!e^{-q\hat r/8} \,d\hat r\notag\\
   &\leq C.\label{star}
\end{align}
Similarly, 
\[
  \eps^2\bigl[|\pt^2_{\xi} g|+{\color{blue}|\pt^2_{\xi\eta} g|}+
    |\pt^2_{\eta} g|\bigr]
    \leq C\eps^{-1}e^{q\hat\xi}\, (1+\hat r^{-1})\,K_{0,1}(q\hat r)
    \leq C\eps^{-1}\hat r^{-2} e^{-q\hat r/8},
\]
so
\begin{gather}\label{star2}
   \iint_{\hat\Omega_1\setminus B(0,0;\hat\rho)}\!
   \bigl[|\pt^2_{\xi} g|+{\color{blue}|\pt^2_{\xi\eta} g|}+
    |\pt^2_{\eta} g|\bigr]
   \,\bigl(\eps^2d\hat\xi \,d\hat\eta\bigr)
   \leq C\eps^{-1}\!\int_{\hat\rho}^{\infty}\!\!\hat r^{-1}e^{-q\hat r/8} \,d\hat r
   \leq C\eps^{-1}\ln(2+\hat\rho^{-1}).
\end{gather}
Furthermore, for an arbitrary ball $\hat B_{\hat\rho}$ of radius $\hat\rho$
in the coordinates $(\hat\xi,\hat\eta)$, we get
\beq\label{ball_Omega1}
\iint_{\hat\Omega_1\cap \hat B_{\hat\rho}}\!\bigl[|\pt_\xi g|+|\pt_\eta g|+|g|
\bigr]
\,\bigl(\eps^2d\hat\xi \,d\hat\eta\bigr)
\leq C\int_0^{\hat\rho}\!\!e^{-q\hat r/8} \,d\hat r\leq C\min\{\hat\rho,1\}.
\eeq

(ii) Next consider $(\hat\xi,\hat\eta)\in\hat\Omega_2$.
In this subdomain, it is convenient to rewrite the integrals in terms of
$(\hat \xi,t)$, where
\beq\label{t_def}
t:=\hat\xi^{-1/2}\,\hat\eta\qquad
\mbox{so}\qquad
\hat\xi^{-1/2}\,d\hat \eta=dt,
\qquad
\hat r-\hat\xi=\frac{\hat\eta^2}{\hat r+\hat\xi}\le t^2.
\eeq
Note that $q\hat r\ge q\ge \frac12\alpha$ in $\hat\Omega_2$, so
$K_{0,1}(q\hat r)\le C (q\hat r)^{-1/2}e^{-q\hat r}$
by the second bound in \eqref{K_01_bound}.
We also note that $\hat\xi\le\hat r\le \sqrt{17}\hat\xi$
in $\hat\Omega_2$
so $\hat{r}-\hat\xi=\hat{\eta}^2/(\hat{r}+\hat\xi)\ge c_0{\hat{\eta}^2}/{\hat\xi}=
c_0 t^2$,
where $c_0:=(1+\sqrt{17})^{-1}$.
Consequently
$e^{-q(\hat{r}-\hat\xi)}\le 
e^{-q c_0t^2}$,
so
\beq\label{Q_def}
e^{q\hat{\xi}}\,K_{0,1}(q\hat r)
\le C Q
\quad\mbox{for}\;\;(\hat\xi,\hat\eta)\in\hat\Omega_2,
\qquad\mbox{where}\quad
Q:=\hat\xi^{-1/2}\, e^{-q c_0t^2}
\eeq
and
\beq\label{Q_int}
   \int_\R(1+|t|+t^2+|t|^3+t^4)\,Q\,d\hat\eta
   \le C \int_\R (1+|t|+t^2+|t|^3+t^4)\,e^{-q c_0t^2}\,dt
   \le C.
\eeq

We now claim that for $g$ and its derivatives in $\hat\Omega_2$ one has
\begin{subequations}\label{bounds15}
\begin{align}
   \label{bound0}
      \eps^2| g|
      &\le C\,\eps\, Q,\\
   \label{bound1}
      \eps^2|\pt_\eta g|
      &\le C\xi^{-1/2}\, |t|\, Q,\\
   \label{bound2a}
      \eps^2|\pt^2_{\eta} g|
      &\le C\,\eps^{-1}\hat\xi^{-1}\, [t^2+1]\, Q.
%
\end{align}
Here
\eqref{bound0} is straightforward, and
\eqref{bound1} immediately follows from \eqref{g_eta} as
$|\hat\eta|/\hat r\le |\hat\eta|/\hat\xi=\xi^{-1/2}|t|$.
The next bound \eqref{bound2a}
is obtained from \eqref{g2_eta_eta} using $\hat\eta^2/\hat r^2\le \xi^{-1}t^2$
and $|\hat\eta^2-\hat\xi^2|/\hat r^3\le\hat r^{-1}\le\hat\xi^{-1}$.

Furthermore, we claim that in $\hat\Omega_2$ one also has
\beqa
\label{bound2}
\eps^2(\eps \hat r|\pt_\xi g|+|\pt_q g|)&\le&
C\,\eps [t^2+1]\, Q,\\
\label{bound2b}
\eps^2|\pt_\xi g|&\le&
C\,\hat\xi^{-1}\, [t^2+1]\, Q,\\
\label{bound5}
\eps^2(\eps\hat r|\pt^2_{\xi\eta} g|)&\le&  C\,
\hat\xi^{-1/2}\, |t|\,[t^2+1] Q,\\
\label{bound3}
\eps^2(\eps \hat r|\pt^2_{\xi q} g|)&\le&
C\,\eps [t^4+1]\, Q+ q^{-1}\eps^2(\eps \hat r|\pt_\xi g|),
\\
   \label{bound2_new}
      \eps^2|\pt^2_{\xi} g|
      &\le& C\,\eps^{-1}\hat\xi^{-2}\, [t^4+1]\, Q.
%
\eeqa
\end{subequations}
To get \eqref{bound2}, we
combine \eqref{g_xi} and \eqref{g_q} with the observation that
\begin{subequations}\label{K01_eta_r}
\begin{align}
\label{K01_eta_r_a}
|K_\nu(q\hat{r})-\frac{\hat\xi}{\hat{r}}K_\mu(q\hat{r})|
  &=\bigl|1-\frac{\hat\xi}{\hat{r}}+\mathcal{O}(\hat{r}^{-1})\bigr|K_1(q\hat r)
  &\mbox{for}\;\hat r\ge 1,\\
\label{K01_eta_r_b}
  &\le C\hat r^{-1}[t^2+1]K_1(q\hat r)
  &\mbox{for}\;\hat\xi\ge 1,
\end{align}
\end{subequations}
where $\nu,\mu=0,1$.
Note that \eqref{K01_eta_r_a} and \eqref{K01_eta_r_b}
are easily verified using \eqref{K_0_via_K_1} and
$\hat r-\hat\xi\le t^2$ from \eqref{t_def}, respectively.
The bound \eqref{bound2b} follows from
the bound for $\pt_\xi g$ in \eqref{bound2} as 
$\hat r^{-1}\le\hat\xi^{-1}$.
We now proceed to
\eqref{bound5}, which
is obtained from \eqref{g2_xi_eta} again using
$|\hat\eta|/\hat r\le \xi^{-1/2}|t|$
and then \eqref{K01_eta_r_b} and $\hat\xi/\hat r\le 1$.
Next, one gets \eqref{bound3}  from \eqref{g2_xi_q} using
$\{2K_0(q\hat{r})+{\ts\frac1{ q\hat r}}K_1(q\hat{r})\}
=2K_1(q\hat{r})[1+\mathcal{O}(\hat r^{-2})]$
(which follows from \eqref{K_0_via_K_1})
and then
$(\hat r-\hat\xi)^2\le t^4$.
The final bound \eqref{bound2_new} is derived in a similar manner
by employing \eqref{K_0_via_K_1} to
rewrite the term in square-brackets of \eqref{g2_xi_xi}
as
$
\bigl[q \bigl(1-\frac{\hat\xi}{\hat r}\bigr)^2
-\frac{3\hat\eta^2}{2\hat r^3}
+\mathcal{O}(\hat r^{-2})
\bigr]K_1(q\hat r).
$
Thus all the bounds \eqref{bounds15} are now established.

Combining the obtained estimates \eqref{bounds15} with \eqref{Q_int} yields
\begin{align}
  \iint_{\hat\Omega_2}\!\bigl[|g|+\eps^{1/2}|\pt_\eta g|+\eps\hat r|\pt_\xi g|+|\pt_q g|
                              +\eps^{1/2}\eps\hat r|\pt^2_{\xi\eta} g|+\eps\hat r|\pt^2_{\xi q} g|
                              +\eps|\pt^2_{\xi} g|&+{\color{blue} \eps|\pt^2_{\xi\eta} g|}
                              \bigr]
                       \,\bigl(\eps^2d\hat\xi \,d\hat\eta\bigr)
   \notag\\
   \label{int_Omega2_main}
   \leq C\int_1^{2/\eps}\!\![\eps+\eps^{1/2}\hat\xi^{-1/2}+\hat\xi^{-2}+{\color{blue}\hat\xi^{-3/2}}
   ] \,d\hat\xi
   \leq C.&
\end{align}
{\color{blue}Note that here we also used
$\eps^2|\pt^2_{\xi\eta} g|\le  C\,\eps^{-1}
\hat\xi^{-3/2}\, |t|\,[t^2+1] Q$, which follows from \eqref{bound5} combined with $\hat r\simeq\hat \xi$.}
Similarly,
combining  \eqref{bound2a},\,\eqref{bound2b} with \eqref{Q_int} yields
\beq\label{int_Omega2_aux}
\iint_{\hat\Omega_2}\!\bigl[|\pt_\xi g|+\eps|\pt^2_{\eta} g|\bigr]
\,\bigl(\eps^2d\hat\xi \,d\hat\eta\bigr)
\leq C\int_1^{2/\eps}\!\!\hat\xi^{-1} \,d\hat\xi\leq C(1+|\ln\eps|).
\eeq
Furthermore,
by \eqref{bound1},\,\eqref{bound2b},
for an arbitrary ball $\hat B_{\hat\rho}$ of radius $\hat\rho$
in the coordinates $(\hat\xi,\hat\eta)$, we get
\beq\label{ball_Omega2}
\iint_{\hat\Omega_2\cap \hat B_{\hat\rho}}\!\bigl[|\pt_\xi g|+|\pt_\eta g|+|g|\bigr]
\,\bigl(\eps^2d\hat\xi \,d\hat\eta\bigr)
\leq C\int_1^{1+\hat\rho}\!\![\hat\xi^{-1}+\hat\xi^{-1/2}+\eps] \,d\hat\xi\leq C\hat\rho.
\eeq
%

To complete the proof, we now
recall that $\hat\Omega\subset\hat\Omega_1\cup\hat\Omega_2$ and
combine
estimates \eqref{star} and \eqref{star2} (that involve integration over $\hat\Omega_1$)
with \eqref{int_Omega2_main} and \eqref{int_Omega2_aux}, 
which yields the desired bounds \eqref{g_L1}--\eqref{g_eta_x_L1}
and  
{\color{blue}\eqref{g_xi2_L1}--\eqref{g_eta2_L1}}.
To get the latter three bounds we also used the observation
that the ball $B(x,y;\rho)$ 
in the coordinates $(\xi,\eta)$
becomes the ball $B(0,0;\hat\rho)$ of radius $\hat\rho=\eps^{-1}\rho$
in the coordinates $(\hat\xi,\hat\eta)$.
%
The remaining assertion \eqref{g_eta_ball}
is obtained by combining \eqref{ball_Omega1} with \eqref{ball_Omega2} and
noting that an arbitrary ball $B(x',y';\rho)$ of radius $\rho$
in the coordinates $(\xi,\eta)$
becomes a ball $\hat B_{\hat\rho}$ of radius $\hat\rho=\eps^{-1}\rho$
in the coordinates $(\hat\xi,\hat\eta)$.
\end{proof}

\begin{rem}
   The first bound \eqref{g_L1}
   of Lemma~\ref{lem:g0_bounds} can be also obtained by noting
   that $I_g(\xi):=\int_\R\!g\,d\eta$ satisfies
   the differential equation $[-\eps\pt^2_{\xi}+ 2q \pt_\xi]I_g=\delta(x-\xi)$
   (this follows from  an equation of type \eqref{eq:Green_adj_const} for $g$)
   and the conditions $I_g(-\infty)=0$ and $I_g(x)=(2q)^{-1}$. From this, one
   can 
   easily deduce that $\int_0^1 I_g(\xi)\le C$,
   which yields \eqref{g_L1} in view of $g>0$.
\end{rem}

Our next result shows that
for $x\ge 1$, one gets stronger bounds for $g$ and its derivatives.
These bounds involve
the weight function
\beq\label{lambda_def}
   \lambda:=e^{2q(x-1)/\eps}
\eeq
and show that, although
$\lambda$ is exponentially large in $\eps$,
this is compensated by the smallness of $g$ and its derivatives.

\begin{lem}\label{lem:g_lmbd_bounds}
   Let $(x,y)\in[1,3]\times\R$
   and $0<\frac12\alpha\le q\le C$.
   Then for the function  $g=g(x,y;\xi,\eta;q)$
   of \eqref{eq:def_g0}
   and the weight $\lambda$ of \eqref{lambda_def},
   one has the following bounds
   \begin{subequations}\label{g_bounds_lmbd_}
   \begin{align}
       \|([1+\eps\hat r_{[x]}]\,\lambda g)(x,y;\cdot;q)\|_{1\,;\Omega}
      &\leq C\eps,\label{g_L1_mu}\\
       \|(\lambda\, \pt_\xi g)(x,y;\cdot;q)\|_{1\,;\Omega}
      +\|(\lambda\,\pt_q g)(x,y;\cdot;q)\|_{1\,;\Omega}
      &\leq C,\label{g_L1_mu2}\\
       \|([1+\eps^{1/2}\hat r_{[x]}]\,\lambda\,\pt_\eta g)(x,y;\cdot;q)\|_{1\,;\Omega}
      +\eps^{1/2}\|(\eps\hat r_{[x]}\,\lambda\,\pt_{\xi\eta}^2 g)(x,y;\cdot;q)\|_{1\,;\Omega}
      &\leq C,\label{g_L1_mu3}\\
       \|\hat r_{[x]}\, \pt_q(\lambda\, g)(x,y;\cdot;q)\|_{1\,;\Omega}
      +\|\eps \hat r_{[x]}\,\pt_q(\lambda\,\pt_\xi g)(x,y;\cdot;q)\|_{1\,;\Omega}
      &\le C,\label{g_L1_mu4}
   \end{align}
   and for any ball $B(x',y';\rho)$ of radius $\rho$ centred at any
   $(x',y')\in[0,1]\times\R$, one has
   \begin{align}
      \|(\lambda\, g)(x,y;\cdot;q)\|_{1,1\,; B(x',y';\rho)}
            &\leq C\eps^{-1}\rho,
      \label{g_L1_mu5}
      \intertext{while for the ball $B(x,y;\rho)$ of radius $\rho$ centred at $(x,y)$, one has }
      %
      \color{blue}\|(\lambda\,\pt^{2-j}_\xi \pt^j_{\eta} g)(x,y;\cdot;q)\|_{1\,;\Omega\setminus B(x,y,\rho)}
            &\color{blue}\leq C\eps^{-1}\ln(2+\eps/\rho)\quad\mbox{for}\;\;j=0,1,2.\label{g_L1_mu6}
   \end{align}
   \end{subequations}
   Furthermore, with the differential operators \eqref{D_ops}, we have
   \begin{subequations}\label{g_lmb_bounds}
   \begin{align}
      \label{D_g_lmbd}
      \hspace*{-0.2cm}
       \|\pt_x(\lambda  g)(x,y;\cdot;q)\|_{1\,;\Omega}
      +\|D_y(\lambda  g)(x,y;\cdot;q)\|_{1\,;\Omega}
      +\|D_\eta(\lambda  g)(x,y;\cdot;q)\|_{1\,;\Omega}
      &\leq C,\\
      \label{D_pt_lmbd}
       \|\eps \hat r_{[x]}\,D_y(\lambda\, \pt_\xi g)(x,y;\cdot;q)\|_{1\,;\Omega}
      +\|\eps \hat r_{[x]}\,D_\eta\pt_x(\lambda g)(x,y;\cdot;q)\|_{1\,;\Omega}
      &\leq C\eps^{-1/2}.
   \end{align}
   \end{subequations}
\end{lem}

\begin{proof}
Throughout the proof we use the notation $A=A(x):=(x-1)/\eps\ge 0$.
Then \eqref{lambda_def} becomes $\lambda=e^{2qA}$.
We partially imitate the proof of Lemma~\ref{lem:g0_bounds}.
Again $d\xi\,d\eta=\eps^2\,d\hat\xi\,d\hat\eta$,
but now  $(\hat\xi,\hat\eta)\in\hat\Omega=\eps^{-1}(-x,1-x)\times\R
\subset(-3/\eps,-A)\times\R$.
%
So $\hat\xi<-A\le 0$ immediately yields
\beq\label{lambda_exp}
\lambda\, e^{q \hat\xi}=e^{2q(A-|\hat\xi|)}\,e^{q|\hat\xi|}\le e^{q|\hat\xi|}.
\eeq

Consider the domains
\[
   \hat\Omega'_1:=\bigl\{|\hat\xi|<1+{\ts\frac14|\hat\eta|},\;\hat\xi<-A\bigr\},
   \qquad
   \hat\Omega'_2:=\bigl\{|\hat\xi|>\max\{\ts1,\frac14|\hat\eta|\}
   ,\;-3/\eps<\hat\xi<-A\bigr\}.
\]
As $\hat\Omega\subset\hat\Omega_1'\cup\hat\Omega_2'$
 for any $x\in[1,3]$, we estimate integrals over these two domains
 separately.

 (i) Let $(\hat\xi,\hat\eta)\in\hat\Omega'_1$.
Then $|\hat\xi|\le1+\frac14\hat r$
so, by \eqref{lambda_exp}, one has
$\lambda\, e^{q \hat\xi}\le e^{q(1+\hat r/4)}$.
The first line in \eqref{star0} remains valid, but
now we combine it with
\beq\label{K_0_1_exp}
\lambda \,e^{q\hat\xi}\, (1+\hat r+\hat r^2)\,
K_{0,1}(q\hat r)
\le C\hat r^{-1} \, e^{-q\hat r/8}
\eeq
(which is obtained similarly to the second line in \eqref{star0}).
This leads to a version of \eqref{star}
that involves the weight $\lambda$:
\beq
\iint_{\hat\Omega_1'}\lambda\,\bigl
[ (1+\hat r)(\eps^{-1}|g|+|\pt_\xi g|+|\pt_\eta g|+|\pt_q g|+|\pt^2_{\xi q} g|)
+\eps\hat r|\pt^2_{\xi\eta} g|\bigr]
\,\bigl(\eps^2d\hat\xi \,d\hat\eta\bigr)
\leq C.
\label{star_lambda}
\eeq
In a similar manner, we obtain versions of estimates \eqref{star2}
and \eqref{ball_Omega1},
that also involve the weight $\lambda$:
\begin{gather}\label{star2_lambda}
   \iint_{\hat\Omega_1'\setminus B(0,0;\hat\rho)}\!
      \lambda\,\bigl[{\color{blue}|\pt_{\xi\eta}^2 g|}+|\pt_\eta^2 g|\bigr]
      \,\bigl(\eps^2d\hat\xi \,d\hat\eta\bigr)
         \leq C\eps^{-1}\!\int_{\hat\rho}^{\infty}\!\!\hat r^{-1}e^{-q\hat r/8} \,d\hat r
         \leq C\eps^{-1}\ln(2+\hat\rho^{-1}),
\end{gather}
\beq\label{ball_Omega1_lambda}
\iint_{\hat\Omega_1'\cap \hat B_{\hat\rho}}\!\lambda\,\bigl[|\pt_\xi g|+|\pt_\eta g|+|g|
\bigr]
\,\bigl(\eps^2d\hat\xi \,d\hat\eta\bigr)
\leq C\int_0^{\hat\rho}\!\!e^{-q\hat r/8} \,d\hat r\leq C\min\{\hat\rho,1\},
\eeq
where $\hat B_{\hat\rho}$ is
an arbitrary ball of radius $\hat\rho$
in the coordinates $(\hat\xi,\hat\eta)$.
Furthermore, \eqref{star_lambda} combined with
$|\pt_q(\lambda\, g)|\le \lambda(2A|g|+|\pt_q g|)$
and
$|\pt_q(\lambda\,\pt_\xi g)|
\le
\lambda(2A|\pt_\xi g|+|\pt^2_{\xi q} g|)$
and then with $A\le 2/\eps$
yields
\beq
\iint_{\hat\Omega_1'}\, \hat r\,\bigl
[  |\pt_q(\lambda\, g)|
+\eps|\pt_q(\lambda\,\pt_\xi g)|\bigr]
\,\bigl(\eps^2d\hat\xi \,d\hat\eta\bigr)
\leq C.
\label{star_lambda2}
\eeq

(ii) Now consider $(\hat\xi,\hat\eta)\in\hat\Omega'_2$.
In this subdomain (similarly to $\hat\Omega_2$
in the proof of Lemma~\ref{lem:g0_bounds})
one has $|\hat\xi|\le \hat r\le\sqrt{17}|\hat\xi|$
and
$c_0 t^2\le\hat r-|\hat\xi|\le t^2$, where
 $t:=|\hat\xi|^{-1/2}\,\hat\eta$
(compare with \eqref{t_def}).
We also introduce a new barrier $Q$ such that
\beq\label{Q_def2}
e^{q\hat{\xi}}\,K_{0,1}(q\hat r)
\le C Q
\quad\mbox{for}\;\;(\hat\xi,\hat\eta)\in\hat\Omega'_2,
\quad\mbox{where}\;\;
Q:=\lambda ^{-1}\,e^{2q(A-|\hat\xi|)}\,\bigl\{|\hat\xi|^{-1/2}\, e^{-q c_0t^2}\bigr\},
\eeq
(compare with \eqref{Q_def}).
Note that the inequality in \eqref{Q_def2} is obtained similarly to the one in
\eqref{Q_def}, as \eqref{lambda_exp} implies
$e^{q\hat{\xi}}\,K_{0,1}(q\hat r)=
\lambda^{-1}\, e^{2q(A-|\hat\xi|)}\,\{e^{q|\hat\xi|}\,K_{0,1}(q\hat r)\}$.

With the new definition \eqref{Q_def2} of $Q$,
the bounds \eqref{bound0}--\eqref{bound2a} remain valid in $\hat\Omega_2'$
only with $\hat\xi$ replaced by $|\hat\xi|$.
Note that the bounds \eqref{bound2}--\eqref{bound3} are not valid in $\hat\Omega_2'$,
(as they were obtained using $\hat r-\hat\xi\le t^2$,
which is not the case for $\hat\xi<0$).
Instead,
we claim that in $\hat\Omega'_2$ one has
\begin{subequations}\label{bounds_lambda}
\beqa
\label{bound2b_lmb}
\eps^2|\pt_\xi g|&\le&
C\, Q,\\
\label{bound2_lmb}
\eps^2|\pt_q g|&\le&
C\,\eps |\hat\xi|\, Q,\\
\label{bound5_lmb}
\eps^2(\eps\hat r|\pt_{\xi\eta} g|)&\le&  C\,
|\hat\xi|^{1/2}\, |t|\, Q,
\\
\label{bound2a_lmb}
\eps^2(|\pt_q(\lambda\, g)|+\eps|\pt_q(\lambda\,\pt_\xi g)|)&\le&
C\,\eps\, \lambda\,[(|\hat\xi|-A)+t^2+1]\, Q.
%
\eeqa
\end{subequations}
Here \eqref{bound2b_lmb} immediately follows from \eqref{g_xi} as
$|\hat\xi|/\hat r\le 1$.
The bound \eqref{bound2_lmb} is obtained from \eqref{g_q} in a similar way,
also using $\hat r\le\sqrt{17}|\hat\xi|$.
The next bound \eqref{bound5_lmb},
is deduced from \eqref{g2_xi_eta}
using
$\eta=|\hat\xi|^{1/2} t$ and again
$|\hat\xi|/\hat r\le 1$, and also $\hat r+1\le 2\hat r$.

To establish \eqref{bound2a_lmb}, note
that
$\pt_q(\lambda\, g)=\lambda[2A\,g+\pt_{q}g]$ and
$\pt_q(\lambda\,\pt_\xi g)=\lambda[2A\,\pt_\xi g+\pt^2_{\xi q}g]$.
Using \eqref{K_0_via_K_1}, \eqref{K01_eta_r_a}
and
$\{2K_0(q\hat{r})+{\ts\frac1{ q\hat r}}K_1(q\hat{r})\}
=2K_1(q\hat{r})[1+\mathcal{O}(\hat r^{-2})]$
(which follows from \eqref{K_0_via_K_1}),
 one can rewrite the definition of $g$
and relations \eqref{g_q}, \eqref{g_xi}, \eqref{g2_xi_q} as
\begin{align*}
   g
   &={\ts \frac{1}{2\pi\eps}}\,e^{q\hat\xi}\,
     \Bigl[ 1+\mathcal{O}(\hat r^{-1})\Bigr]K_1(q\hat{r}),\\
   \partial_q g
   &={\ts \frac{1}{2\pi\eps}}\,e^{q\hat\xi}\,
     \Bigl[ -(\hat r+|\hat\xi|)+\mathcal{O}(1)\Bigr]K_1(q\hat{r}),\\
   \partial_\xi g
   &={\ts \frac{q}{2\pi\eps^2}}\,e^{q\hat\xi}\,\,\hat r^{-1}
     \Bigl[ (\hat r+|\hat\xi|)+\mathcal{O}(1)\Bigr]K_1(q\hat{r}),\\
   \partial^2_{\xi q} g
   &={\ts \frac{q}{2\pi\eps^2}}\,e^{q\hat\xi}\,\,\hat r^{-1}
     \Bigl[-(\hat{r}+|\hat\xi|)^2+\mathcal{O}(1)\Bigr]K_1(q\hat{r})
     +q^{-1}\partial_{\xi} g.
\end{align*}
Next note that
\[
   S:=(\hat{r}+|\hat\xi|)-2A= 2(|\hat\xi|-A)+(\hat r-|\hat\xi|)
   \le  2(|\hat\xi|-A)+t^2.
\]
Consequently, a calculation shows that
\begin{align*}
   \lambda^{-1}\,\pt_q(\lambda\, g)
   &={\ts \frac{1}{2\pi\eps}}\,e^{q\hat\xi}\,
      \Bigl[ -S+\hat r^{-1}\mathcal{O}(A+\hat r)\Bigr]K_1(q\hat{r}),\\
   \lambda^{-1}\,\pt_q(\lambda\,\pt_\xi g)
   &={\ts \frac{q}{2\pi\eps^2}}\,e^{q\hat\xi}\,
     \Bigl[-S\,\hat r^{-1}(\hat{r}+|\hat\xi|)
           +\hat r^{-1}\mathcal{O}(A+1)\Bigr]K_1(q\hat{r})
      +q^{-1}\partial_{\xi} g.
\end{align*}
In view of $\hat r^{-1}(A+\hat r+1)\le C$ and $\hat r^{-1}(\hat r+|\hat\xi|)\le 2$,
and also \eqref{bound2b_lmb}, the final bound \eqref{bound2a_lmb} in
\eqref{bounds_lambda} follows.

 Next, note that
 \eqref{Q_int} is valid with $Q$ replaced by
 the multiplier $\{e^{q|\hat\xi|}\,K_{0,1}(q\hat r)\}$
 from the current definition \eqref{Q_def2} of $Q$.
Combining this observation with
the bounds \eqref{bound0}--\eqref{bound2a} and
\eqref{bound2b_lmb}--\eqref{bound5_lmb},
and also with $\hat r\le \sqrt{17}|\hat\xi|$,
yields
\begin{align}
   \iint_{\hat\Omega_2'}\!\!\lambda\,\bigl[
   (\eps^{-1}+\hat r)|g|+|\pt_\xi g|+(1+\eps^{1/2}\hat r)|\pt_\eta g|
      &+|\pt_q g|+\eps^{1/2}(\eps\hat r|\pt^2_{\xi\eta} g|)
      \notag\\
     &\hspace{1cm}{}+{\color{blue}\eps|\pt^2_{\xi\eta} g|} +\eps|\pt^2_{\eta} g|\bigr]
      \,\bigl(\eps^2d\hat\xi \,d\hat\eta\bigr)\notag\\
   \label{int_Omega2_main_lmbd}
   \leq C\int_{-3/\eps}^{-\max\{A,1\}}\!
   \bigl[1+\eps|\hat\xi|+|\hat\xi|^{-1/2}&+(\eps|\hat\xi|)^{1/2}+|\hat\xi|^{-1}\bigr]\,
   e^{2q(A-|\hat\xi|)} \,d\hat\xi
   \leq C.
\end{align}
{\color{blue}Here we also employed the version $\eps^2(\eps|\pt_{\xi\eta} g|)\le  C\,|\hat\xi|^{-1/2}\, |t|\, Q$
of \eqref{bound5_lmb}
(obtained, using $|\hat \xi|\simeq\hat r$).
}
Similarly, from \eqref{bound2a_lmb} combined with
$\hat r\eps\le \sqrt{17}|\hat\xi|\,\eps\le 3\sqrt{17}$, one gets
\beq\label{Omega2_prime}
\iint_{\hat\Omega_2'}\!
\hat r\bigl[|\pt_q(\lambda\, g)|+\eps|\pt_q(\lambda\,\pt_\xi g)|\bigr]
\,\bigl(\eps^2d\hat\xi \,d\hat\eta\bigr)
\le
 C\int_{-3/\eps}^{-\max\{A,1\}}\!\!
\bigl[(|\hat\xi|-A)+1\bigr]\,
e^{2q(A-|\hat\xi|)} \,d\hat\xi\le C.
\eeq
Furthermore,
by \eqref{bound1},\,\eqref{bound2b_lmb},
for an arbitrary ball $\hat B_{\hat\rho}$ of radius $\hat\rho$
in the coordinates $(\hat\xi,\hat\eta)$, we get
\beq\label{ball_Omega2_lambda}
\iint_{\hat\Omega'_2\cap \hat B_{\hat\rho}}\!
\lambda\,\bigl[|\pt_\xi g|+|\pt_\eta g|+|g|\bigr]
\,\bigl(\eps^2d\hat\xi \,d\hat\eta\bigr)
\leq C\int^{-\max\{A,1\}}_{-\max\{A,1\}-\hat\rho}
\bigl[1+|\hat\xi|^{-1/2}\bigr]\,
e^{2q(A-|\hat\xi|)}\,d\hat\xi\leq C\hat\rho.
\eeq

To complete the proof of \eqref{g_bounds_lmbd_}, we now
recall that $\hat\Omega\subset\hat\Omega_1'\cup\hat\Omega_2'$ and
 combine
estimates \eqref{star_lambda}, \eqref{star2_lambda}, \eqref{star_lambda2}
(that involve integration over $\hat\Omega_1'$)
with \eqref{int_Omega2_main_lmbd}, \eqref{Omega2_prime},
which yields the desired bounds \eqref{g_L1_mu}--\eqref{g_L1_mu4}
and the bound for $\color{blue}\pt^2_{\xi\eta} g$ and  $\pt^2_\eta g$ in \eqref{g_L1_mu6}.
To get the latter bound we also used the observation
that the ball $B(x,y;\rho)$ 
in the coordinates $(\xi,\eta)$
becomes the ball $B(0,0;\hat\rho)$ of radius $\hat\rho=\eps^{-1}\rho$
in the coordinates $(\hat\xi,\hat\eta)$.
The bound for $\pt^2_\xi g$ in (\ref{g_L1_mu6})
follows as
$\pt^2_\xi g=-\pt^2_\eta g+\frac{2q}\eps\,\pt_\xi g$
for $(\xi,\eta)\neq(x,y)$.
%
The remaining assertion \eqref{g_L1_mu5}
is obtained by combining
\eqref{ball_Omega1_lambda} with \eqref{ball_Omega2_lambda}
and
noting that an arbitrary ball $B(x',y';\rho)$ of radius $\rho$
in the coordinates $(\xi,\eta)$
becomes a ball $\hat B_{\hat\rho}$ of radius $\hat\rho=\eps^{-1}\rho$
in the coordinates $(\hat\xi,\hat\eta)$.
Thus we have established all the bounds \eqref{g_bounds_lmbd_}.

We now proceed to the proof of the bounds \eqref{g_lmb_bounds}.
Note that $\pt_x g=-\pt_\xi g$ and
$\pt_y g=-\pt_\eta g$.
Combining these with \eqref{g_L1_mu2} and
the bound for $\|\lambda\,\pt_\eta g\|_{1\,;\Omega}$ in \eqref{g_L1_mu3},
yields
$\|\lambda\, \pt_x g\|_{1\,;\Omega}
+\|\lambda\, D_y g\|_{1\,;\Omega}+\|\lambda\, D_\eta g\|_{1\,;\Omega}\le C$.
Now, combining
$\pt_x\lambda=2q\eps^{-1}\lambda$ and
$\pt_q\lambda=2A\lambda\le 4\eps^{-1}\lambda$
with \eqref{g_L1_mu}, yields
$\|g\,\pt_x\lambda\|_{1\,;\Omega}+\|g\,D_y\lambda\|_{1\,;\Omega}
+\|g\,D_\eta\lambda\|_{1\,;\Omega}\le C$.
Consequently, we get \eqref{D_g_lmbd}.

To estimate $\eps \hat r_{[x]}\,D_y(\lambda\, \pt_\xi g)$, note that it
involves $\eps \hat r_{[x]}\,\pt_y(\lambda\, \pt_\xi g)
=-\eps \hat r_{[x]}\,\lambda\, \pt^2_{\xi\eta} g$
(as $\lambda$ is independent of $y$ and $\pt_y g=-\pt_\eta g$),
for which we have a bound in \eqref{g_L1_mu3},
and also $\eps \hat r_{[x]}\,\pt_q(\lambda\, \pt_\xi g)$,
for which we have a bound in \eqref{g_L1_mu4}. The desired bound
for $\eps \hat r_{[x]}\,D_y(\lambda\, \pt_\xi g)$
in \eqref{D_pt_lmbd} follows.

For the second bound in \eqref{D_pt_lmbd},
a calculation yields
$\eps \hat r_{[x]}\,D_\eta\pt_x(\lambda g)=
\eps \hat r_{[x]}\,D_\eta(\lambda\, \pt_x g)
+2\hat r_{[x]}\,D_\eta( q \lambda\,g)$.
The first term is estimated similarly to $\eps \hat r_{[x]}\,D_y(\lambda\, \pt_\xi g)$
in \eqref{D_pt_lmbd}.
The remaining term $\hat r_{[x]}\,D_\eta( q \lambda\,g)$
involves
$\hat r_{[x]}\,\pt_\eta( q \lambda\,g)=q\,\hat r_{[x]}\,\lambda\,\pt_\eta g$,
for which we have a bound in \eqref{g_L1_mu3},
and also
$\hat r_{[x]}\,\pt_q( q \lambda\,g)=q \,\hat r_{[x]}\,\pt_q( \lambda\,g)
+\hat r_{[x]}\,\lambda\,g$,
for which we have bounds in \eqref{g_L1_mu4} and \eqref{g_L1_mu}.
Consequently we get the second bound in \eqref{D_pt_lmbd}.
\end{proof}

\begin{lem}\label{lem_1_3}
Under the conditions of Lemma~\ref{lem:g_lmbd_bounds},
for some positive constant $c_1$
 one has
\beq\label{lambda_g_1_3}
\|\lambda g(x,y;\cdot)\|_{2,1\,;[0,\frac13]\times\R}
+\|D_y(\lambda g)(x,y;\cdot)\|_{1,1\,;[0,\frac13]\times\R}
\le  Ce^{-c_1\alpha/\eps}.
\eeq
\end{lem}

\begin{proof}
We imitate the proof of Lemma~\ref{lem:g_lmbd_bounds}, only now
$\xi<\frac13$ or $\hat\xi< (\frac13-x)/\eps\le -\frac23/\eps$.
Thus instead of the subdomains $\hat\Omega_1'$ and $\hat\Omega_2'$
we now consider $\hat\Omega_1''$ and $\hat\Omega_2''$
defined by
$\hat\Omega_k'':=\hat\Omega_k'\cap\{\hat\xi< -(x-\frac13)/\eps\}$.
Thus in $\hat\Omega_1''$
\eqref{K_0_1_exp} remains valid with $q\ge\frac12\alpha$, but now
 $\hat r>\frac23/\eps$.
Therefore, when we integrate over $\hat\Omega_1''$
(instead of $\hat\Omega_1'$),
the integrals of type \eqref{star_lambda}, \eqref{star2_lambda}
become bounded by $Ce^{-c_1\alpha/\eps}$ for any fixed $c_1<\frac1{16}$.
Next, when considering integrals over  $\hat\Omega_2''$
(instead of $\hat\Omega_2'$),
note that $A-|\hat\xi|\le-\frac23/\eps$ so
the quantity $e^{2q(A-|\hat\xi|)}$ in the definition \eqref{Q_def2} of $Q$
is now bounded by $e^{-\frac23\alpha/\eps}$.
Consequently, the integrals of type \eqref{int_Omega2_main_lmbd}
over $\hat\Omega_2''$ also become bounded by
$Ce^{-c_1\alpha/\eps}$.
\end{proof}

\begin{rem}\label{rem_q_var}
   All the estimates of Lemmas~\ref{lem:g0_bounds}, \ref{lem:g_lmbd_bounds}
   and~\ref{lem_1_3} remain valid
   if one sets $q:=\frac12a(x,y)$
   or $q:=\frac12 a(\xi,\eta)$
   in $g$, $\lambda$, and their derivatives
   (after the differentiation is performed).
\end{rem}

\section{Approximations $\bar G$ and $\tilde G$ for Green's function $G$}\label{sec:bounds_green_const}
%
We shall use  two related cut-off functions $\omega_0$ and $\omega_1$
defined by
\beq\label{cut_off}
\omega_0(t) \in C^2(0,1),\quad\omega_0(t)=1\;\mbox{ for }t\le{\ts\frac23},
\quad\omega_0(t)=0\;\mbox{ for }t\ge{\ts\frac56};
\quad
\omega_1(t):=\omega_0(1-t),
\eeq
so $\omega_k(k)=1$, $\omega_k(1-k)=0$
 and $\frac{d^m}{dt^m}\omega_k(0)=\frac{d^m}{dt^m}\omega_k(1)=0$ for $k=0,1$ and $m=1,2$.

Recall that solutions $\bar g$ and $\tilde g$ of the frozen-coefficient equations
\eqref{eq:Green_adj_const} and \eqref{eq:Green_prim_const}
in the domain $\R^2$
are explicitly given by \eqref{bar_tilde_g_def}, \eqref{eq:def_g0}.
Now consider these two equations in some domain $\Omega\subset\R^2$
subject to homogeneous Dirichlet boundary conditions on $\pt\Omega$.
For such problems,
one can employ $\bar g$ and $\tilde g$ to construct solution approximations
using the method of images with an inclusion of the above cut-off
functions.
%
First we construct such solution approximations, denoted by $\bar G$ and $\tilde G$,
for the domain $\Omega=(0,1)\times\R$ (in Section~\ref{ssec:bounds_green_const_1}),
then for our domain of interest $\Omega=(0,1)^2$
(in Section~\ref{ssec:bounds_green_const_2}).

Note that although $\bar G$ and $\tilde G$ are constructed as solution approximations
for the frozen-coefficient equations,
we shall see in Section~\ref{sec:main_proof} that they, in fact,
provide approximations to the Green's function $G$ for our original
variable-coefficient problem.

  \subsection{Approximations  $\bar G$ and $\tilde G$ for the domain $\Omega=(0,1)\times\R$}\label{ssec:bounds_green_const_1}
%
As outlined earlier in this Section~\ref{sec:bounds_green_const},
for the domain $\Omega=(0,1)\times\R$, we define $\bar G$ and $\tilde G$ by
%
%
%
\beq\label{bar_tilde_G}
\bar G(x,y;\xi,\eta):=\bar{\mathcal G}
\bigr|_{q=\frac12a(x,y)},
\qquad
\tilde G(x,y;\xi,\eta):=\tilde {\mathcal G}
\bigr|_{q=\frac12a(\xi,\eta)},
\eeq
\vspace{-0.7cm}
\begin{subequations}\label{bar_tilde_G20}
\begin{align}
   \label{bar_G_def}
   \hspace{-0.3cm}
   \bar {\mathcal G}(x,y;\xi,\eta;q)
      &\!:=\!{\ts\frac{1}{2\pi\eps}}
             e^{q\hat\xi_{[x]}}\!\Bigl\{\!
                \left[K_0(q\hat r_{[x]})\!-\!K_0(q\hat r_{[-x]})\right]
                \!-\!\left[K_0(q\hat r_{[2-x]})\!-\!K_0(q\hat r_{[2+x]})\right]\!\omega_1(\xi)
            \!\Bigr\},\!\!\\
   \label{tilde_G_def}
   \hspace{-0.3cm}
   \tilde {\mathcal G}(x,y;\xi,\eta;q)
      &\!:=\!{\ts\frac{1}{2\pi\eps}}
             e^{q\hat\xi_{[x]}}\!\Bigl\{\!
             \left[K_0(q\hat r_{[x]})\!-\!K_0(q\hat r_{[2-x]})\right]
             \!-\!\left[K_0(q\hat r_{[-x]})\!-\!K_0(q\hat r_{[2+x]})\right]\!\omega_0(x)
             \!\Bigr\}.\!\!
\end{align}
\end{subequations}
Note that $\bar G\bigr|_{\xi=0,1}=0$ and
$\tilde G\bigr|_{x=0,1}=0$
(the former observation
follows from $r_{[x]}=r_{[-x]}$ at $\xi=0$,
and $r_{[x]}=r_{[2-x]}$ and $r_{[-x]}=r_{[2+x]}$ at $\xi=1$).
We shall see shortly (see Lemma~\ref{lem_tilde_bar_w}) that
$\bar L^*_{\xi\eta}\bar G\approx L^*_{\xi\eta}G$
and $\tilde L_{xy}\tilde G\approx L_{xy} G$;
in this sense $\bar G$ and $\tilde G$ give approximations for $G$.

Rewrite the definitions of $\bar{\mathcal G}$ and $\tilde{\mathcal G}$
using the notation
\begin{subequations}
\begin{align}
   \label{g_x_brackets}
   &g_{[x]}:= g(x,y;\xi,\eta;q)={\ts\frac{1}{2\pi\eps}}\, e^{q\hat\xi_{[x]}}\,
              K_0(q\hat{r}_{[x]}),\\
   \label{lmb_p_def}
   &\lambda^{\pm}:=e^{2q(1\pm x)/{\eps}},\qquad\quad p:=e^{-2qx/{\eps}},
\end{align}
\end{subequations}
and the observation that
\[
   {\ts\frac{1}{2\pi\eps}}\,e^{q\hat\xi_{[x]}}\,K_0(q\hat r_{[d]})
   =e^{q(d-x)/\eps}\,g_{[d]}\qquad\mbox{for}\;\; d=\pm x, 2\pm x.
\]
They yield
\begin{subequations}\label{bar_G_tilde_g}
\beqa
\label{bar_G_g}
\bar{\mathcal G}(x,y;\xi,\eta;q)&=&
\left[g_{[x]}-p\, g_{[-x]}\right]
-\left[\lambda^- g_{[2-x]}
-p\,\lambda^{\!+} g_{[2+x]}\right]\omega_1(\xi),
\\
\label{tilde_G_g}
\tilde {\mathcal G}(x,y;\xi,\eta;q)&=&
\left[g_{[x]}-\lambda^- g_{[2-x]}\right]
-\left[p\, g_{[-x]}
-p\,\lambda^{\!+} g_{[2+x]}\right]\omega_0(x).
\eeqa
\end{subequations}
Note that $\lambda^\pm$ is obtained by replacing $x$ by $2\pm x$
in the definition \eqref{lambda_def} of $\lambda$.

In the next lemma, we estimate the functions
\begin{equation}\label{tilde_w_def}
   \bar \phi(x,y;\xi,\eta)=\bar L^*_{\xi\eta}\bar G-L^*_{\xi\eta}G,
   \qquad
   \tilde \phi(x,y;\xi,\eta):=\tilde L_{xy}\tilde G-L_{xy}G.
\end{equation}

\begin{lem}\label{lem_tilde_bar_w}
Let $(x,y)\in\Omega=(0,1)\times\R$. Then
for the functions $\bar\phi$ and $\tilde \phi$ of \eqref{tilde_w_def}, one has
\beq\label{tilde_bar_w}
\|\bar \phi(x,y;\cdot)\|_{1,1\,;\Omega}
+\|\pt_y\bar \phi(x,y;\cdot)\|_{1\,;\Omega}+
\|\tilde \phi(x,y;\cdot)\|_{1,1\,;\Omega}\le
 Ce^{-c_1\alpha/\eps}\le C.
\eeq
Furthermore, for $\bar\phi$ we also have
 \begin{gather}\label{eq:bar_phi_boundary}
   \bar\phi(x,y;\xi,\eta)|_{(\xi,\eta)\in\pt\Omega}=0.
 \end{gather}
\end{lem}

\begin{proof}
(i)
First we prove the desired assertions for $\bar\phi$.
By \eqref{bar_tilde_G},
throughout this part of the proof we set $q=\frac12a(x,y)\ge\frac12\alpha$.
Recall that $\bar g$ solves the differential equation \eqref{eq:Green_adj_const}
with the operator $\bar L^*_{\xi\eta}$.
Comparing the explicit formula for $\bar g$ in \eqref{bar_tilde_g_def}
with the notation \eqref{g_x_brackets}
implies that
$\bar L^*_{\xi\eta}g_{[d]}=\delta(\xi-d)\delta(\eta-y)$.
So, by \eqref{eq:Green_adj}, $\bar L^*_{\xi\eta}g_{[x]}=L_{\xi\eta}^*G$,
and also $\bar L^*_{\xi\eta}g_{[d]}=0$ for $d=-x,2\pm x$ and all $(\xi,\eta)\in\Omega$
as $(d,y)\not\in\Omega$.
Now, by \eqref{bar_G_g}, we conclude that
$\bar\phi=-\bar L^*_{\xi\eta}[\omega_1(\xi) \bar {\mathcal G}_2]$
where $\bar {\mathcal G}_2:=\lambda^- g_{[2-x]}-p\,\lambda^{\!+} g_{[x+2]}$,
and $\bar L^*_{\xi\eta}\bar {\mathcal G}_2=0$ for $(\xi,\eta)\in\Omega$.

From these observations,
$\bar\phi=2\eps\omega_1'(\xi)\pt_\xi \bar {\mathcal G}_2
+[\eps\omega_1''(\xi)-2q\omega_1'(\xi)]\bar {\mathcal G}_2$.
The definition \eqref{cut_off} of $\omega_1$ implies that
$\bar\phi$
vanishes at $\xi=0$ and for $\xi\ge\frac13$.
This implies the desired assertion \eqref{eq:bar_phi_boundary}.
Furthermore, we now get
\[
    \|\bar \phi(x,y;\cdot)\|_{1,1\,;\Omega}
   +\|\pt_y\bar \phi(x,y;\cdot)\|_{1\,;\Omega}
   \le C\bigl(
            \|\bar {\mathcal G}_2(x,y;\cdot)\|_{2,1\,;[0,\frac13]\times\R}
            +\|D_y\bar {\mathcal G}_2(x,y;\cdot)\|_{1,1\,;[0,\frac13]\times\R}
        \bigr).
\]
Combining this with the bounds \eqref{lambda_g_1_3} for the terms $\lambda^\pm g_{[2\pm x]}$
of $\bar {\mathcal G}_2$,
and the observation that
 $|D_yp|\le C|\pt_q p|\le C$ and $\pt_\xi p=\pt_\eta p=0$, yields
 our assertions for $\bar\phi$ in \eqref{tilde_bar_w}.

(ii)
Now we prove the desired estimate \eqref{tilde_bar_w} for $\tilde \phi$.
By \eqref{bar_tilde_G}, throughout this part of the proof we set $q=\frac12a(\xi,\eta)\ge\frac12\alpha$.
Comparing the notation \eqref{g_x_brackets} with
the explicit formula for $\tilde g$ in \eqref{bar_tilde_g_def}, we rewrite
\eqref{eq:Green_prim_const}
as $\tilde L_{xy}g_{[x]}=\delta(x-\xi)\delta(y-\eta)$.
So
$\tilde L_{xy}g_{[x]}=L_{xy}G$, by~\eqref{eq:Green_prim}.
Next, for each value $d=-x,2\pm x$
respectively set $s=-\xi,\mp(2-\xi)$. Now
by \eqref{eq:def_g0}, one has
$\hat r_{[d]}
=\sqrt{(s-x)^2+(\eta-y)^2}/\eps$
so $g(x,y;s,\eta;q)=\frac{1}{2\pi\eps}e^{q(s-x)/\eps}K_0(q\hat r_{[d]})$.
Note that $\tilde L_{xy}g(x,y;s,\eta;q)=\delta(x-s)\delta(y-\eta)$
and none of our three values of $s$ is in $[0,1]$ (i.e. $\delta(s-x)=0$).
Consequently,
$\tilde L_{xy}[e^{q\hat\xi_{[x]}}K_0(q\hat r_{[d]})]=0$
for all $(x,y)\in\Omega$.
Comparing \eqref{tilde_G_def} and \eqref{tilde_G_g}, we now conclude that
$\tilde\phi=-\tilde L_{xy}[\omega_0(\xi) \tilde {\mathcal G}_2]$
where $\tilde{\mathcal G}_2:=p\, g_{[-x]}-p\,\lambda^{\!+} g_{[x+2]}$
and $\tilde L_{xy}\tilde{\mathcal G}_2=0$ for $(x,y)\in\Omega$.

From these observations,
$\tilde\phi
=2\eps\omega_0'(x)\pt_x \tilde{\mathcal G}_2
+[\eps\omega_0''(x)+2q\omega_0'(x)]\tilde{\mathcal G}_2$.
As the definition \eqref{cut_off} of $\omega_0$ implies that $\tilde\phi$
vanishes for $x\le\frac23$, we have
\[
   \|\tilde\phi(x,y;\cdot)\|_{1,1\,;\Omega}
   \le C
   \max_{\stackrel{(x,y)\in[\frac23,1]\times\R}{k\,=\,0,1}}\|\pt^k_x
   \tilde {\mathcal G}_2(x,y;\cdot)\|_{1,1\,;\Omega}\,.
\]
Here $\tilde G_2$ is smooth and has no singularities
for $x\in[\frac23,1]$
(because $\hat r_{[2+x]}\ge \hat r_{[-x]}\ge\frac23\eps^{-1}$ for $x\in[\frac23,1]$).
Note that
$\|\pt^k_x g_{[-x]}\|_{1,1\,;\Omega}\le C \eps^{-2}$,
and
$\|\pt^k_x (\lambda^{\!+}g_{[2+x]})\|_{1,1\,;\Omega}\le C \eps^{-2}$
(these two estimates are similar to the ones in
 Lemmas~\ref{lem:g0_bounds} and~\ref{lem:g_lmbd_bounds},
but easier to deduce as they are not sharp).
We combine these two bounds with $|\pt^k_x \pt^m_{\xi}\pt_\eta^n p|\le C \eps^{-2}p
= C \eps^{-2}e^{-2qx/{\eps}}$ for $k,\, m+n\le 1$.
As for $x\ge\frac23$ we enjoy the bound $e^{-2qx/{\eps}}\le e^{-\frac23\alpha/{\eps}}\le
C\eps^4 e^{-\frac12\alpha/{\eps}}$, the desired estimate for $\tilde\phi$ follows.
\end{proof}

\begin{lem}\label{lem:tilde_bar_G}
   Let the function  $R=R(x,y;\xi,\eta)$ be such that
   $|R|\le C\min\{\eps\hat r_{[x]},1\}$.
   The functions $\bar G$ and $\tilde G$ of \eqref{bar_tilde_G}, \eqref{bar_G_tilde_g}
   satisfy
   \begin{subequations}
      \begin{align}
         \|\bar G(x,y;\cdot)\|_{1\,;\Omega}+\|\tilde G(x,y;\cdot)\|_{1\,;\Omega}
            &\le C,\label{G_bar_t_1}\\
         \|\pt_\xi\bar G(x,y;\cdot)\|_{1\,;\Omega}
            &\le C(1+|\ln\eps|),\label{bar_G_xi}\\
         \|\pt_\eta\bar G(x,y;\cdot)\|_{1\,;\Omega}
            &\le C\eps^{-1/2},\label{bar_G_eta}\\
         \|(R\,\pt_\xi\bar G)(x,y;\cdot)\|_{1\,;\Omega}
         +\eps^{1/2}\|(R\,\pt^2_{\xi\eta}\bar G)(x,y;\cdot)\|_{1\,;\Omega}
            &\le C,\label{bar_G_xi_eta}
   \intertext{and for any ball $B(x',y';\rho)$ of radius $\rho$ centred at any
              $(x',y')\in[0,1]\times\R$, one has}
         |\bar G(x,y;\cdot)|_{1,1\,;B(x',y';\rho)\cap\Omega}
            &\le C\eps^{-1}\rho,\label{bar_G_1_1_B}
   \intertext{while for the ball $B(x,y;\rho)$ of radius $\rho$ centred at $(x,y)$,
              we have}
         \|\pt^2_{\xi}\bar G(x,y;\cdot)\|_{1\,;\Omega\setminus B(x,y;\rho)}+{\color{blue}\|\pt^2_{\xi\eta}\bar G(x,y;\cdot)\|_{1\,;\Omega\setminus B(x,y;\rho)}}
            &\le C\eps^{-1}\ln(2+\eps/\rho),\label{bar_G_xi_xi_2}\\
         \|\pt^2_{\eta}\bar G(x,y;\cdot)\|_{1\,;\Omega\setminus B(x,y;\rho)}
            &\le C\eps^{-1}(\ln(2+\eps/\rho)+|\ln\eps|).\label{bar_G_eta_eta_2}
   \intertext{Furthermore, we have}
         \|\pt_y\bar G(x,y;\cdot)\|_{1\,;\Omega}
         +\|(R\,\pt^2_{\xi y}\bar G)(x,y;\cdot)\|_{1\,;\Omega}
            &\le C\eps^{-1/2},\label{bar_G_y}\\
         \|\pt_\eta\tilde G(x,y;\cdot)\|_{1\,;\Omega}
            &\le C\eps^{-1/2},\label{G_tilde_eta}\\
         \int_0^1\!\bigl(\|(R\,\pt^2_{x\eta}\tilde G)(x,y;\cdot)\|_{1\,;\Omega}+
         \|\pt_x\tilde G(x,y;\cdot)\|_{1\,;\Omega}\bigr)\,dx
            &\le C\eps^{-1/2}.\label{G_tilde_eta_x}
      \end{align}
   \end{subequations}
\end{lem}

\begin{proof}
First, note that
 $\hat r_{[-x]}\ge \hat r_{[x]}$
and $\hat r_{[2\pm x]}\ge \hat r_{[x]}$ for all $(\xi,\eta)\in\Omega$, therefore
\beq\label{R_ext}
|R|\le C\,\min\bigl\{\eps\hat r_{[x]},\,\eps\hat r_{[-x]},\,\eps\hat r_{[2-x]},\,\eps\hat r_{[2+x]},
\,1\bigr\}.
\eeq

Note also that in view of Remark~\ref{rem_q_var},
 all bounds of Lemma~\ref{lem:g0_bounds}  apply to
the components $g_{[\pm x]}$
and all bounds of Lemma~\ref{lem:g_lmbd_bounds} apply to
the components $\lambda^\pm g_{[2\pm x]}$
of $\bar{\mathcal G}$ and $\tilde{\mathcal G}$ in \eqref{bar_G_tilde_g}.

\underline{\it Asterisk 
notation.}
In some parts of this proof, when discussing derivatives of $\bar {\mathcal G}$,
we shall use the  {notation} $\bar {\mathcal G}^\ast$ 
prefixed by some differential operator, e.g., $\pt_x \bar {\mathcal G}^\ast$.
This will mean that the differential operator is applied only to the terms of the
type $g_{[d\pm x]}$, e.g., $\pt_x \bar {\mathcal G}^\ast$ is obtained by replacing
each of the four terms $g_{[d\pm x]}$ in the definition \eqref{bar_G_g}
 of $\bar {\mathcal G}$
by $\pt_x g_{[d\pm x]}$ respectively.

(a)
The first desired estimate \eqref{G_bar_t_1}
follows from the bound \eqref{g_L1} for $g_{[\pm x]}$
and the bound \eqref{g_L1_mu} for $\lambda^\pm g_{[2\pm x]}$
combined with $|p|\le 1$ and $|\omega_{0,1}|\le 1$
(in fact, the bound for $\bar G$ can obtained
by imitating the proof of Lemma~\ref{lem_G_L1}).

(b)(c)(d)
Rewrite \eqref{bar_G_g} as
\[
   \bar {\mathcal G}=\bar {\mathcal G}_1-\omega_1(\xi)\bar {\mathcal G}_2,
   \quad
   \mbox{where}\quad
   \bar {\mathcal G}_1:=g_{[x]}-p\, g_{[-x]},
   \quad
   \bar {\mathcal G}_2:=\lambda^- g_{[2-x]}-p\,\lambda^{\!+} g_{[2+x]}.
\]
As $q=\frac12a(x,y)$ in $\bar G$
(i.e. $p$ and $\lambda^\pm$ in $\bar G$ do not involve $\xi,\,\eta$),
one gets
\beq\label{NN_new_aux}
   \pt_\xi \bar G=\pt_\xi\bar {\mathcal G}^\ast-\omega_1'(\xi)\bar {\mathcal G}_2,
   \qquad
   \pt_\eta \bar G=\pt_\eta\bar {\mathcal G}^\ast,
   \qquad
   \pt^2_{\xi\eta} \bar G=\pt^2_{\xi\eta}\bar {\mathcal G}^\ast
   -\omega_1'(\xi)\pt_{\eta}\bar {\mathcal G}^\ast_2.
\eeq
%
Now the desired estimate \eqref{bar_G_xi}
follows from the bound \eqref{g_xi_L1} for $\pt_\xi g_{[\pm x]}$,
the bound \eqref{g_L1_mu2} for $\lambda^\pm \,\pt_\xi g_{[2\pm x]}$,
and the bound \eqref{g_L1_mu} for $\lambda^\pm g_{[2\pm x]}$.
Similarly, \eqref{bar_G_eta}
follows from the bound \eqref{g_eta_L1} for $\pt_\eta g_{[\pm x]}$,
and  the bound \eqref{g_L1_mu3} for
$\lambda^\pm \pt_{\eta}g_{[2\pm x]}$.

The next desired estimate \eqref{bar_G_xi_eta} is deduced using
\[
   |R\,\pt_\xi \bar G|\le |R \,\pt_\xi\bar {\mathcal G}_1^\ast|
   +C|\pt_\xi\bar {\mathcal G}_2^\ast|+C|\bar {\mathcal G}_2|,
   \qquad
   |R\,\pt^2_{\xi\eta} \bar G|\le|R\,\pt^2_{\xi\eta}\bar {\mathcal G}^\ast|
   +C|\pt_{\eta}\bar {\mathcal G}^\ast_2|.
\]
Here, in view of \eqref{R_ext}, the term $R \,\pt_\xi\bar {\mathcal G}_1^\ast$
is estimated using the bound \eqref{g_xi_R_L1} for
$\eps\hat r_{[\pm x]} \pt_\xi g_{[\pm x]}$,
while
the term $R\,\pt^2_{\xi\eta}\bar {\mathcal G}^\ast$
is estimated using the bound \eqref{g_eta_x_L1} for
$\eps\hat r_{[\pm x]}\pt^2_{\xi\eta}g_{[\pm x]}$
and the bound \eqref{g_L1_mu3} for
$\lambda^\pm \eps\hat r_{[2\pm x]}\pt^2_{\xi\eta}g_{[2\pm x]}$.
The remaining terms $\pt_\xi\bar {\mathcal G}_2^\ast$, $\bar {\mathcal G}_2$
and $\pt_{\eta}\bar{\mathcal G}^\ast_2$
appear in  $\pt_\xi\bar G$ and $\pt_\eta\bar G$,
so have been bounded when obtaining \eqref{bar_G_xi}, \eqref{bar_G_eta}.

(e) The next assertion \eqref{bar_G_1_1_B} is proved similarly to
\eqref{bar_G_xi} and \eqref{bar_G_eta}, only
using the bound \eqref{g_eta_ball} for $g_{[\pm x]}$
and the bound \eqref{g_L1_mu5} for $\lambda^\pm g_{[2\pm x]}$.

(f)(g)
As $q=\frac12a(x,y)$ in $\bar G$, then
 {\color{blue}$\pt^2_{\xi} \bar G=\pt_\xi^2\bar {\mathcal G}^\ast
 -2\omega_1'(\xi)\pt_\xi\bar {\mathcal G}_2-\omega_1''(\xi)\bar {\mathcal G}_2$,
 while
   $\pt_{\xi\eta} \bar G$ is taken from \eqref{NN_new_aux},}
and $\pt^2_{\eta} \bar G=\pt_\eta^2\bar {\mathcal G}^\ast$.
Now,
the assertions  \eqref{bar_G_xi_xi_2} and \eqref{bar_G_eta_eta_2}
immediately follow from the bounds  {\color{blue}\eqref{g_xi2_L1}--\eqref{g_eta2_L1} for
 $\pt^{2-j}_{\xi}\pt^j_{\eta}g_{[\pm x]}$
combined with
the bound \eqref{g_L1_mu6} for
 $\lambda^\pm \pt^{2-j}_{\xi}\pt^j_{\eta}g_{[2\pm x]}$, as well as the bounds \eqref{g_L1_mu}  for $\lambda^\pm g_{[2\pm x]}$,  \eqref{g_L1_mu2}
 for $\lambda^\pm \pt_{\xi}g_{[2\pm x]}$, and \eqref{g_L1_mu3} for
$\lambda^\pm \pt_{\eta}g_{[2\pm x]}$.}

(h)
We again have $q=\frac12a(x,y)$ in $\bar G$, so
using the operator $D_y$ of \eqref{D_ops},
one gets
\begin{multline*}
   \pt_y \bar G
   = D_y\bigl[g_{[x]}-p \,g_{[-x]}\bigr]^\ast
     -\omega_1(\xi)\,
     \bigl[D_y(\lambda^-  g_{[2-x]})-p\, D_y(\lambda^{\!+} g_{[2+x]})\bigr]\\
   -{\ts\frac12} \pt_y a(x,y)\cdot \pt_q p\cdot\bigl[  g_{[-x]}
   -\omega_1(\xi)\lambda^{\!+} g_{[2+x]}\bigr],
\end{multline*}
where $|\pt_q p|\le C$ by \eqref{lmb_p_def}
(and we used the previously defined notation $\ast$).
%
Now,  $\pt_y \bar G$ is estimated
using the bound \eqref{D_g} for $D_yg_{[\pm x]}$
and the bound \eqref{D_g_lmbd} for
$D_y(\lambda^\pm  g_{[2\pm x]})$.
For the term $g_{[-x]}$  in $\pt_y \bar G$  we use the bound \eqref{g_L1},
and for the term  $\lambda^{\!+} g_{[2+x]}$
the bound \eqref{g_L1_mu}.
Consequently, one gets the desired bound \eqref{bar_G_y} for $D_y\bar G^\ast$.

To estimate $R\,\pt^2_{\xi y} \bar G$, 
a calculation shows that
\begin{multline*}
\pt^2_{\xi y} \bar G
=(D_y\pt_\xi) \bigl[g_{[x]}-p \,g_{[-x]}\bigr]^\ast
   -\omega_1(\xi)\,
   \bigl[D_y(\lambda^- \pt_\xi g_{[2-x]})-p\, D_y(\lambda^{\!+}\pt_\xi g_{[2+x]})\bigr]\\
   -{\ts\frac12} \pt_y a(x,y)\cdot \pt_q p\cdot\bigl[\pt_\xi  g_{[-x]}
   -\omega_1(\xi)\lambda^{\!+}\pt_\xi g_{[2+x]}\bigr]
   -\omega_1'(\xi)\pt_y\bar G_2,
\end{multline*}
where $\bar G_2:=\bar{\mathcal G}_2\bigr|_{q=a(x,y)/2}$.
The assertion \eqref{bar_G_y}  for $R\,\pt^2_{\xi y} \bar G$ is now deduced as follows.
In view of \eqref{R_ext},
we  employ the bound \eqref{D_pt_g} for the terms $\eps\hat r_{[\pm x]}D_y\pt_\xi g_{[\pm x]}$
and the bound \eqref{D_pt_lmbd} for the terms
$\eps \hat r_{[2\pm x]}\,D_y(\lambda^\pm\, \pt_\xi g_{[2\pm x]})$.
For the remaining terms (that appear in the second line)
we use $|R|\le C$ and $|\pt_q p|\le C$.
Then we combine the bound \eqref{g_xi_L1} for $\pt_\xi  g_{[-x]}$
and the bound \eqref{g_L1_mu2} for $\lambda^{\!+}\pt_\xi g_{[2+x]}$.
The term $\pt_y\bar G_2$ is a part of $\pt_y\bar G$, which was estimated above,
so for $\pt_y\bar G_2$
we have the same bound as for $\pt_y\bar G$ in \eqref{bar_G_y}.
This observation completes the proof of the bound
for $R\,\pt^2_{\xi y} \bar G$
in \eqref{bar_G_y}.

(i)(j)
We now proceed to estimating derivatives of $\tilde G$,
so $q=\frac12a(\xi,\eta)$ in this part of the proof.
Let
$\tilde {\mathcal G}^\pm:=g_{[\pm x]}-\lambda^{\mp} g_{[2\mp x]}$.
Then \eqref{tilde_G_g}, \eqref{lmb_p_def} imply that
$\tilde {\mathcal G}=\tilde {\mathcal G}^+-p_0\tilde {\mathcal G}^-$, where
$p_0:=\omega_0(x)p=\omega_0(x)e^{-2qx/\eps}$.
Note that
\[
   D_\eta p_0={\ts\frac12}\pt_\eta a(\xi,\eta)\cdot(-2x/\eps)\,p_0,
   \qquad
   \pt_x p_0=[\omega_0'(x)-(2q/\eps)\omega_0(x)]e^{-2qx/{\eps}}.
\]
Combining this with $|(-2x/\eps)\,p_0|\le C e^{-qx/{\eps}}$
and $q\ge \frac12 \alpha$ yields
\beq\label{p_0}
|D_\eta p_0|\le C,\qquad
\int_0^1\!\bigl(|\pt_x p_0|+|D_\eta \pt_x p_0|\bigr)dx\le
\int_0^1\!\bigl(C\eps^{-1}e^{-\frac12\alpha x/{\eps}}\bigr)dx
\le C.
\eeq

Furthermore, we claim that
\beq\label{tilde_cal_G}
\|\tilde {\mathcal G}^-\|_{1\,;\Omega}\le C,
\qquad
\|\pt_x\tilde {\mathcal G}^\pm\|_{1\,;\Omega}\le C(1+|\ln\eps|),
\qquad
\|D_\eta\tilde {\mathcal G}^\pm\|_{1\,;\Omega}\le C\eps^{-1/2}.
\eeq
Here the first estimate
follows from  the bounds \eqref{g_L1}, \eqref{g_L1_mu} for
the terms $g_{[-x]}$ and $\lambda^{\!+}g_{[2+x]}$.
The estimate for $\pt_x\tilde {\mathcal G}^\pm$ in \eqref{tilde_cal_G} follows from
the bound \eqref{g_x_L1} for $\pt_x g_{[\pm x]}$ and
the bound \eqref{D_g_lmbd} for $\pt_x(\lambda^{\pm}g_{[2\pm x]})$.
Similarly, the estimate for $D_\eta\tilde {\mathcal G}^\pm$
in \eqref{tilde_cal_G}
is obtained using
the bound \eqref{D_g} for $D_\eta g_{[\pm x]}$ and
the bound \eqref{D_g_lmbd} for $D_\eta(\lambda^{\pm}g_{[2\pm x]})$.

Next, a calculation shows that
\[
   \pt_\eta\tilde G=D_\eta\tilde {\mathcal G}^+
   -p_0\,D_\eta\tilde {\mathcal G}^--D_\eta p_0\cdot\tilde {\mathcal G}^-,
   \qquad
   \pt_x\tilde G=\pt_x\tilde {\mathcal G}^+
   -p_0\,\pt_x\tilde {\mathcal G}^--\pt_x p_0\cdot\tilde {\mathcal G}^-.
\]
Combining these with \eqref{p_0}, \eqref{tilde_cal_G} yields
\eqref{G_tilde_eta}
and the bound for $\pt_x\tilde G$ in \eqref{G_tilde_eta_x}.

To establish the estimate for $R\,\pt^2_{x\eta}\tilde G$
in \eqref{G_tilde_eta_x}, note that
\[
   \pt^2_{x\eta}\tilde G=D_\eta\pt_x \tilde {\mathcal G}^+
   -p_0\cdot D_\eta\pt_x \tilde {\mathcal G}^-
   -\pt_x p_0\cdot D_\eta \tilde {\mathcal G}^--\pt_\eta p_0\cdot \pt_x \tilde {\mathcal G}^-
   -D_\eta\pt_{x} p_0\cdot  \tilde {\mathcal G}^-.
\]
In view of \eqref{R_ext}, \eqref{p_0} and \eqref{tilde_cal_G},
it now suffices to show that
$\|R\,D_\eta\pt_x \tilde {\mathcal G}^\pm\|_{1\,;\Omega}\le C\eps^{-1/2}$.
This latter estimate follows from the bound \eqref{D_pt_g}
for the terms $\eps\hat r_{[\pm x]}\, D_\eta\pt_x g_{[\pm x]}$
and the bound \eqref{D_pt_lmbd} for the terms
$\eps\hat r_{[\pm x]}\, D_\eta\pt_x (\lambda^\pm g_{[2\pm x]})$.
This completes the proof of \eqref{G_tilde_eta_x}.
\end{proof}

\subsection{Approximations $\bar G$ and $\tilde G$
for the domain $\Omega=(0,1)^2$}\label{ssec:bounds_green_const_2}
We now define approximations,
denoted by $\bar G_{\mbox{\tiny$\Box$}}$ and $\tilde G_{\mbox{\tiny$\Box$}}$,
for our original square domain $\Omega=(0,1)^2$. For this,
we use the approximations $\bar G$ and $\tilde G$ of \eqref{bar_tilde_G}, \eqref{bar_tilde_G20}
for the domain $(0,1)\times\R$
and again employ the method of images with an inclusion of the cut-off
functions of \eqref{cut_off} as follows:
\begin{subequations}\label{eq:def_tilde_bar_G}
\begin{align}\label{eq:def_tilde_bar_G_a}
   \bar G_{\mbox{\tiny$\Box$}}(x,y;\xi,\eta)
      &:=  \bar G(x,y;\xi,\eta)
         -\omega_0(\eta)\,\bar G(x,y;\xi,-\eta)
         -\omega_1(\eta)\,\bar G(x,y;\xi,2-\eta),\\
   \tilde G_{\mbox{\tiny$\Box$}}(x,y;\xi,\eta)
      &:=  \tilde G(x,y;\xi,\eta)
         -\omega_0(y)\,\tilde G(x,-y;\xi,\eta)
         -\omega_1(y)\,\tilde G(x,2-y;\xi,\eta).
\end{align}
\end{subequations}
Then $\bar G_{\mbox{\tiny$\Box$}}\bigr|_{\xi=0,1}=0$
and $\tilde G_{\mbox{\tiny$\Box$}}\bigr|_{x=0,1}=0$
(as this is valid for $\bar G$ and $\tilde G$, respectively),
and furthermore, by \eqref{cut_off}, we have
$\bar G_{\mbox{\tiny$\Box$}}\bigr|_{\eta=0,1}=0$
and $\tilde G_{\mbox{\tiny$\Box$}}\bigr|_{y=0,1}=0$.
\begin{rem}\label{rem:strip_to_square}
   Lemmas~\ref{lem_tilde_bar_w} and \ref{lem:tilde_bar_G}
   of the previous section remain valid if
   $\Omega$ is understood as $(0,1)^2$, and
   $\bar G$ and $\tilde G$
   are
   replaced by $\bar G_{\mbox{\tiny$\Box$}}$ and $\tilde G_{\mbox{\tiny$\Box$}}$,
   respectively,
   in the definition \eqref{tilde_w_def} of $\bar\phi$ and $\tilde\phi$
   and in the lemma statements.

   This is shown by imitating the proofs of these two lemmas.
   We leave out the details and only note that the application of the method of images
   in the $\eta$- ($y$-) direction
   is relatively straightforward  as an inspection of \eqref{eq:def_g0}
   shows that in this direction, the fundamental solution $g$
   is symmetric and exponentially decaying away from the singular point.
   %
\end{rem}

\begin{rem}[Neumann conditions along characteristic boundaries]\label{rem:def_tilde_bar_G}
\color{blue}
Suppose that in \eqref{eq:Lu_a}, the Dirichlet boundary conditions at the top and bottom boundaries are replaced by homogeneous Neumann  conditions.
Then the solution will allow a representations of type \eqref{eq:sol_prim} (and a version of representation \eqref{eq:sol_adj} will be valid) via the corresponding Green's function $G^N$, which is defined as in \eqref{eq:Green_adj} and \eqref{eq:Green_prim}, only with Neumann conditions
at the top and bottom boundaries. An inspection of the proofs shows that \textbf{all our main results} remain valid for this case (see, e.g., Remark~\ref{rem_int_G_Nbc}, and, in particular, part (v)
in the proof of Theorem~\ref{thm:main} in \S\ref{sec:main_proof}).
Note also that when constructing approximations
$\bar G^N_{\mbox{\tiny$\Box$}}$ and $\tilde G^N_{\mbox{\tiny$\Box$}}$
for $G^N$, one needs to replace $-\omega_0$ and $-\omega_1$
in \eqref{eq:def_tilde_bar_G}
by respectively $+\omega_0$ and $+\omega_1$ (as is standard in the method of images when dealing with Neumann boundary conditions).
\end{rem}
As $\bar G_{\mbox{\tiny$\Box$}}$ and $\tilde G_{\mbox{\tiny$\Box$}}$
in the domain $\Omega=(0,1)^2$
enjoy the same properties as $\bar G$ and $\tilde G$ in the domain
$(0,1)\times\R$,
we shall sometimes skip the subscript {\tiny$\Box$} when there is no ambiguity.

\section{Green's function for the original problem
in $\Omega=(0,1)^2$.
Proof of Theorem~\ref{thm:main}}\label{sec:main_proof}
   We are now ready to establish our main result, Theorem~\ref{thm:main},
   for the original variable-coefficient problem \eqref{eq:Lu} in the domain
   $\Omega=(0,1)^2$.
   In Section~\ref{sec:bounds_green_const},
   we have already obtained various bounds
  for the approximations $\tilde G_{\mbox{\tiny$\Box$}}$
   and $\bar G_{\mbox{\tiny$\Box$}}$ of $G$ in $\Omega=(0,1)^2$.
   So now we consider
   the two  functions  $\tilde v$ and $\bar v$ given by
   \[
      \tilde v(x,y;\xi,\eta):=[G-\tilde G_{\mbox{\tiny$\Box$}}](x,y;\xi,\eta),
      \qquad
      \bar v(x,y;\xi,\eta)=[G-\bar G_{\mbox{\tiny$\Box$}}](x,y;\xi,\eta).
   \]
   Throughout this section, we shall skip the subscript {\tiny$\Box$}
   as we always deal with the domain $\Omega=(0,1)^2$.

   Note that, by \eqref{tilde_w_def}, we have
   $L_{xy}\tilde v=L_{xy}[G-\tilde G]
     = [\tilde L_{xy}-L_{xy}]\tilde G-\tilde\phi$,
   and similarly
   $L^*_{\xi\eta}\bar v=L^*_{\xi\eta}[G-\bar G]
     = [\bar L^*_{\xi\eta}-L^*_{\xi\eta}]\bar G-\bar\phi$.
  Consequently, the functions $\tilde v$ and $\bar v$ are solutions of the following problems:
   \begin{subequations}
      \begin{align}
         L_{xy} \tilde v(x,y;\xi,\eta)
           &= \tilde h(x,y;\xi,\eta)\;\;\mbox{for}\;(x,y)\in\Omega,
              \quad \tilde v(x,y;\xi,\eta)=0\;\;\mbox{for}\; (x,y)\in\pt\Omega,\label{prob_tilde_v}\\
         L^*_{\xi\eta} \bar v(x,y;\xi,\eta)
           &= \bar h(x,y;\xi,\eta)\;\;\mbox{for}\;(\xi,\eta)\in\Omega,
              \quad \bar v(x,y;\xi,\eta)=0\;\;\mbox{for}\; (\xi,\eta)\in\pt\Omega.\label{prob_bar_v}
      \end{align}
   \end{subequations}
   Here the right-hand sides are given by
   \begin{subequations}\label{tilde_bar_h_new}
      \begin{align}
         \tilde h(x,y;\xi,\eta)
           &:= \pt_x\{R\,\tilde G\}(x,y;\xi,\eta)
               -b(x,y)\,\tilde G(x,y;\xi,\eta)
               -\tilde \phi(x,y;\xi,\eta),\label{tilde_bar_h}\\
         \bar h(x,y;\xi,\eta)
           &:= \{R\,\pt_\xi\bar G\}(x,y;\xi,\eta)
               -b(\xi,\eta)\,\bar G(x,y;\xi,\eta)
               -\bar \phi(x,y;\xi,\eta),\label{bar_h}
      \end{align}
   \end{subequations}
   where
   \begin{gather}\label{R_def}
      R(x,y;\xi,\eta):=a(x,y)-a(\xi,\eta),\qquad\mbox{so}\;\;
      |R|\le C\min\{\eps\hat r_{[x]},1\}.
   \end{gather}
   Applying the solution representation formulas \eqref{eq:sol_prim} and
   \eqref{eq:sol_adj} to problems \eqref{prob_tilde_v} and \eqref{prob_bar_v},
   respectively, one gets
   \begin{subequations}
      \begin{align}
         \tilde v(x,y;\xi,\eta)
           &= \iint_\Omega G(x,y;s,t)\,\tilde h(s,t;\xi,\eta)\,ds\, dt,\label{tilde_v}\\
         \bar v(x,y;\xi,\eta)
           &= \iint_\Omega G(s,t;\xi,\eta)\,\bar h(x,y;s,t)\,ds\,dt.\label{bar_v}
      \end{align}
   \end{subequations}
We now proceed to the {proof of Theorem~\ref{thm:main}}.
\smallskip

\begin{proof}
(i)   First we establish \eqref{eq:thm:G_eta}.
   Note that, by the bounds \eqref{G_tilde_eta} and \eqref{bar_G_y} for $\pt_\eta\tilde G$
   and $\pt_y\bar G$, respectively, it suffices to show that
   $ \|\pt_\eta\tilde v(x,y;\cdot)\|_{1\,;\Omega}
    +\|\pt_y\bar v(x,y;\cdot)\|_{1\,;\Omega}\le C\eps^{-1/2}$.

   Applying $\pt_\eta$ to \eqref{tilde_v} and $\pt_y$ to \eqref{bar_v}, we arrive at
   \begin{align*}
     \pt_\eta\tilde v(x,y;\xi,\eta)
        &= \iint_\Omega \!G(x,y;s,t)\,\pt_\eta\tilde h(s,t;\xi,\eta)\,ds\, dt,\\
     \pt_y\bar v(x,y;\xi,\eta)
        &=\iint_\Omega\!G(s,t;\xi,\eta)\,\pt_y\bar h(x,y;s,t)\,ds\,dt.
   \end{align*}
   From this, a calculation shows that
   \begin{align*}
      \|\pt_\eta\tilde v(x,y;\cdot)\|_{1\,;\Omega}
         &\le \Bigl(\max_{s\in[0,1]}\int_{\R} |G(x,y;s,t)|\,dt\Bigr)\cdot
               \int_0^1\!\!\bigl(\max_{t\in\R}\|\pt_\eta \tilde h(s,t;\cdot)\|_{1\,;\Omega}\bigr)\,ds,\\
      \|\pt_y\bar v(x,y;\cdot)\|_{1\,;\Omega}
         &\le \Bigl(\max_{(s,t)\in\Omega}\|G(s,t;\cdot)\|_{1\,;\Omega}.\Bigr)\cdot
               \|\pt_y\bar h(x,y;\cdot)\|_{1\,;\Omega}.
   \end{align*}
   So, in view of \eqref{G_L1}, to prove \eqref{eq:thm:G_eta}, it remains to show that
   \[
      \int_0^1\!\bigl(\max_{y\in\R}\|\pt_\eta \tilde h(x,y;\cdot)\|_{1\,;\Omega}\bigr)\,dx
        \le C\eps^{-1/2},
      \qquad
      \|\pt_y\bar h(x,y;\cdot)\|_{1\,;\Omega}
         \le C\eps^{-1/2}.
   \]
   These two bounds follow from the definitions \eqref{tilde_bar_h_new},
   \eqref{R_def}
   of $\tilde h$ and $\bar h$, which imply that
   \begin{align*}
      |\pt_\eta \tilde h(x,y;\xi,\eta)|
         &\le |R\,\partial^2_{x \eta}\tilde G|
             +C\bigl( |\pt_x\tilde G|+|\pt_\eta\tilde G|\bigr)
             +|\pt_\eta \tilde\phi|,\\
      |\pt_y\bar h(x,y;\xi,\eta)|
         &\le |R\,\pt^2_{\xi y}\bar G|
             +C\bigl(|\pt_\xi\bar G|+ |\pt_y\bar G|\bigr)
             +|\pt_y\bar\phi|,
   \end{align*}
   combined with the bounds \eqref{tilde_bar_w} for $\bar\phi$, $\tilde\phi$,
   the bounds \eqref{G_tilde_eta}, \eqref{G_tilde_eta_x} for $\tilde G$
   and the bounds \eqref{bar_G_xi}, \eqref{bar_G_y} for $\bar G$.
   Thus we have shown \eqref{eq:thm:G_eta}.
   \smallskip

(ii)
   Next we proceed to obtaining
   the assertions \eqref{eq:thm:G_xi},
   \eqref{eq:thm:G_xixi}
   and
   \eqref{eq:thm:G_etaeta}.
   We claim that to get these three bounds, it suffices to show that
   \begin{subequations}\label{desired}
   \begin{align}
     \mathcal{V}
      :=\max_{(x,y)\in\Omega}\|\pt^2_{\eta}\bar v(x,y;\cdot)\|_{1\,;\Omega}
      &\le C(\eps^{-1}+\eps^{-1/2} \mathcal{W}),\label{bar_v_eta_eta}\\
     \mathcal{W}
      :=\max_{(x,y)\in\Omega}\left(
      \|\partial_\xi {\color{blue}\bar v}(x,y;\cdot)\|_{1\,;\Omega}
      +{\color{blue}\eps\|\pt^2_{\xi}\bar v(x,y;\cdot)\|_{1\,;\Omega}}\right)
      &\le \color{blue}C(1+\eps\mathcal{V}).\label{mathcal_G}
   \end{align}
   \end{subequations}
{\color{blue}   Indeed, there is a sufficiently small constant $c_*$ such that
   for $\eps\le c_*$, combining the bounds \eqref{bar_v_eta_eta},\,\eqref{mathcal_G},
   one gets $\mathcal{W}\le C$, which, combined with
   $\bar v= G-\bar G$ and then
    \eqref{bar_G_xi} and \eqref{bar_G_xi_xi_2} yields \eqref{eq:thm:G_xi} and \eqref{eq:thm:G_xixi}.
   Furthermore, one gets $\mathcal{V}\le C\eps^{-1}$,
      which, combined with \eqref{bar_G_eta_eta_2},
   yields
   \eqref{eq:thm:G_etaeta}. }

   In the simpler non-singularly-perturbed case of $\eps>c_*$,
   we do not need to employ \eqref{desired}.
   Applying the second fundamental inequality \cite{Lady68}
   to \eqref{prob_bar_v} for each fixed $(x,y)$, one gets
   $
   \|\bar v\|_{2,2\,;\Omega}
    \le C_1 (\|\bar v\|_{2\,;\Omega}+\|\bar h\|_{2\,;\Omega})$.
   Here, by the Poincar\'{e} inequality,
   $\|\bar v\|_{2\,;\Omega}\le |\bar v|_{1,2\,;\Omega}$,
   while \eqref{prob_bar_v} combined with \eqref{assmns}  yields
   $|\bar v|_{1,2\,;\Omega}\le C_2\|\bar h\|_{2\,;\Omega}$.
   Note also that $\|\bar h\|_{2\,;\Omega}\le C_3$.
   These observations imply that
   $
    \|\bar v\|_{2,2\,;\Omega}
    \le \bar C$
   (here $\bar C=C_1C_3(C_2+1)$ depends on $c_*$).
   Combining this with \eqref{bar_G_xi}, \eqref{bar_G_xi_xi_2} and \eqref{bar_G_eta_eta_2},
   we again get \eqref{eq:thm:G_xi}, \eqref{eq:thm:G_xixi}   and
   \eqref{eq:thm:G_etaeta}.
   %
   \smallskip

   We shall obtain \eqref{bar_v_eta_eta} in part (iii) and \eqref{mathcal_G}
   in part (iv) below.
   \smallskip

   {\color{blue}
   (iii) To get \eqref{bar_v_eta_eta}, let $\bar V:=\partial_\eta\bar v$.
   The problem \eqref{prob_bar_v} for $\bar v$ implies that
%
\begin{subequations}\label{bar_V}
   \begin{gather}\label{bar_V_eq}
      L^*_{\xi\eta}\bar V(x,y;\xi,\eta)=\bar H(x,y;\xi,\eta)
         \;\;\mbox{for}\;(\xi,\eta)\in\Omega,\\\label{bar_V_bc}
      \bar V(x,y;\xi,\eta)=0\;\;\mbox{for}\;\xi=0,1,
      \quad \partial_\eta\bar V(x,y;\xi,\eta)=0\;\;\mbox{for}\;\eta=0,1.
   \end{gather}
   \end{subequations}
   The homogeneous boundary conditions $\partial_\eta\bar v\bigr|_{\xi=0,1}=0$
   in \eqref{bar_V} immediately follow from $\bar v\bigr|_{\xi=0,1}=0$.
   The homogeneous boundary conditions on the boundary edges
   $\eta=0,1$ are obtained as follows. As $\bar v\bigr|_{\eta=0,1}=0$ so
   $\partial_\xi\bar v\bigr|_{\eta=0,1}=\partial^2_\xi\bar v\bigr|_{\eta=0,1}=0$.
   Combining this with $\bar h\bigr|_{\eta=0,1}=0$
   (for which, in view of Remark~\ref{rem:strip_to_square}, we used \eqref{eq:bar_phi_boundary})
   and the differential equation for $\bar v$ at ${\eta=0,1}$, one
   finally gets
   $\partial^2_\eta\bar v\bigr|_{\eta=0,1}=0$.

   For the right-hand side $\bar H$ in \eqref{bar_V}, a calculation shows
   that 
%
   \[
      \bar H=\partial_\eta\bar h
                -\pt_\eta a(\xi,\eta)\cdot\pt_\xi\bar v-\pt_{\eta} b(\xi,\eta)\cdot\bar v,
   \]
   where, by \eqref{bar_h}, \eqref{R_def},
      \[
      |\pt_\eta\bar h(x,y;\xi,\eta)|
       \le |R\,\pt^2_{\xi\eta} \bar G|
          +C\bigl(|\pt_\xi\bar G|+|\pt_\eta\bar G|+|\bar G|\bigr)
          +|\partial_\eta\bar \phi|.
   \]
Note that a calculation using the bounds \eqref{G_bar_t_1}--
   \eqref{bar_G_xi_eta} for $\bar G$, and the bound
   \eqref{tilde_bar_w} for $\bar\phi$ yields $\|\partial_\eta\bar h(x,y;\cdot)\|_{1\,;\Omega}\le C\eps^{-1/2}$.
Hence,
\beq\label{H_bar_norm}
\|\bar H(x,y;\cdot)\|_{1\,;\Omega}\le C(\eps^{-1/2}+\mathcal{W}),
\eeq
 where we also employed $\bar v= G-\bar G$ and then the bounds \eqref{G_L1},\,\eqref{G_bar_t_1}
   and the definition \eqref{mathcal_G} of~$\mathcal{W}$.
   Now, applying the solution representation formula of type \eqref{eq:sol_adj} to problem \eqref{bar_V}, only with $G$ replaced
   by the corresponding Green's function $G^N$ for the case of Neumann boundary conditions at the top and bottom boundaries (see Remark~\ref{rem:def_tilde_bar_G}), yields
  \beq\label{bar_V_representation}
      \bar V(x,y;\xi,\eta)
       =\iint_\Omega\!
        G^N(s,t;\xi,\eta)\,\bar H(x,y;s,t)\,ds\,dt.
 \eeq
   Note that, in view of Remarks~\ref{rem_int_G_Nbc} and~\ref{rem:def_tilde_bar_G}, one can imitate part (i) of this proof for $G^N$ and derive
    a version of \eqref{eq:thm:G_eta} for $G^N$.
   Thus $\max_{(s,t)\in\Omega}\|\pt_\eta G^N(s,t;\cdot)\|\le C\eps^{-1/2}$,
   so imitating the argument used in part (i) of this proof yields
   \[
      \|\pt^2_{\eta}\bar v(x,y;\cdot)\|_{1\,;\Omega}
      =\|\pt_{\eta}\bar V(x,y;\cdot)\|_{1\,;\Omega}
      \le C \eps^{-1/2}\|\bar H(x,y;\cdot)\|_{1\,;\Omega}.
   \]
   Combining this with \eqref{H_bar_norm} yields the desired bound \eqref{bar_v_eta_eta}.

   Note also that in a similar way, and again using \eqref{bar_V_representation} and then \eqref{H_bar_norm}, one gets
     \beq\label{new_barv_mixed}
      \|\pt^2_{\xi\eta}\bar v(x,y;\cdot)\|_{1\,;\Omega}
      =\|\pt_{\xi}\bar V(x,y;\cdot)\|_{1\,;\Omega}
      \le C(\eps^{-1/2}+\mathcal{W}) \max_{(s,t)\in\Omega}\|\pt_\xi G^N(s,t;\cdot)\|,
\eeq
which will be used in part (v) of this proof to obtain \eqref{eq:thm:G_xieta}.
\smallskip}

   (iv)
   To prove \eqref{mathcal_G},
   rewrite the problem \eqref{prob_bar_v} as a two-point
   boundary-value problem, in which $x$, $y$ and $\eta$ appear as parameters, as follows
   \begin{gather}\label{eq:ode_barv}
      [-\eps\pt^2_{\xi}+a(\xi,\eta)\pt_\xi]\,\bar v(x,y;\xi,\eta)
      =\bar{\bar h}(x,y;\xi,\eta)
      \mbox{ for }\xi\in(0,1),
      \quad \bar v(x,y;\xi,\eta)\bigr|_{\xi=0,1}=0,
   \end{gather}
   where
   \begin{gather}\label{bar_bar_h}
      \bar{\bar h}(x,y;\xi,\eta)
        :=\bar h(x,y;\xi,\eta)+
           \eps\,\pt^2_{\eta}\bar v(x,y;\xi,\eta)
           -b(\xi,\eta)\,\bar v(x,y;\xi,\eta).
   \end{gather}
   Consequently, one can represent $\bar v$ via the Green's function
   $\Gamma=\Gamma(\xi,\eta;s)$
   of the one-dimensional operator $[-\eps\pt^2_{\xi}+a(\xi,\eta)\pt_\xi]$.
   Note that $\Gamma$, for any fixed $\eta$ and $s$,  satisfies the equation
   $[-\eps\pt^2_{\xi}+a(\xi,\eta)\pt_\xi]\Gamma(\xi,\eta;s)=\delta(\xi-s)$
   and the boundary conditions $\Gamma(\xi,\eta;s)\bigr|_{\xi=0,1}=0$.
   Note also that
   \begin{gather}\label{G_1d}
      \int_{0}^{1}\!\! |\partial_\xi \Gamma(\xi,\eta;s)|d\xi
       \le 2\alpha^{-1}
   \end{gather}
   \cite[Lemma 2.3]{And_MM02}; see also
   \cite[(I.1.18)]{RST08}, \cite[(3.10b) and Section~3.4.1.1]{Linss}.

   The solution representation for $\bar v$ via $\Gamma$ is given by
   \[
      \bar v(x,y;\xi,\eta)=\int_0^1\! \Gamma(\xi,\eta;s)\, \bar{\bar h}(x,y;s,\eta)\,ds.
   \]
   Applying $\pt_\xi$ to this representation yields
   \[
      \|\pt_\xi \bar v(x,y;\cdot)\|_{1\,;\Omega}
      \le\left(\max_{(s,\eta)\in \Omega}\int_{0}^{1}\!\! |\pt_\xi \Gamma(\xi,\eta;s)|d\xi\right)
         \cdot\left\|\bar{\bar h}(x,y;\cdot)\right\|_{1\,;\Omega}.
   \]
   In view of \eqref{G_1d}, we now have
   $\| \pt_\xi \bar v\|_{1\,;\Omega} \le 2 \alpha^{-1} \|\bar{\bar h}\|_{1\,;\Omega}$.
   Note that the differential equation \eqref{eq:ode_barv} for $\bar v$
   implies that
   $\eps\norm{\pt_{\xi}^2\bar v}{1;\Omega}
     \leq C(\norm{\pt_{\xi}\bar v}{1;\Omega}+
                    \norm{\bar{\bar h}}{1;\Omega})$.
   So, furthermore, we get
   \[
      \| \pt_\xi \bar v\|_{1\,;\Omega}
      +\eps\| \pt^2_\xi \bar v\|_{1\,;\Omega} \le C \|\bar{\bar h}\|_{1\,;\Omega}.
   \]
{\color{blue}It now remains to show that
   $\|\bar{\bar h}(x,y;\cdot)\|_{1\,;\Omega}\le C+\eps\mathcal{V}$.}
   For this, the definitions \eqref{bar_bar_h} of $\bar{\bar h}$
   and \eqref{bar_v_eta_eta} of $\mathcal{V}$,
   imply that it suffices to prove the two estimates
   \begin{gather}\label{two_bounds}
      \|\bar v(x,y;\cdot)\|_{1\,;\Omega}\le C,
      \qquad
      \|{\bar h}(x,y;\cdot)\|_{1\,;\Omega}\le C.
   \end{gather}
   The first of them follows from $\bar v=G-\bar G$
   combined with \eqref{G_L1} and \eqref{G_bar_t_1}.
   The second is obtained from the definition \eqref{bar_h} of ${\bar h}$
   using \eqref{bar_G_y} for $\|R\,\pt_\xi\bar G\|_{1\,;\Omega}$,
   \eqref{G_bar_t_1} for $\|{\bar G}\|_{1\,;\Omega}$
   and \eqref{tilde_bar_w} for $\|{\bar \phi}\|_{1\,;\Omega}$.
   This completes the proof of \eqref{mathcal_G}, 
   and thus of
   \eqref{eq:thm:G_xi}, \eqref{eq:thm:G_xixi}
   and
   \eqref{eq:thm:G_etaeta}.
   \smallskip

   {\color{blue}(v)
   To obtain the bound \eqref{eq:thm:G_xieta} on $\pt^2_{\xi\eta}G$,
   note that one can imitate the argument of parts (i)--(iv)
 to estimate the relevant derivatives of the  Green's function $G^N$ for the case of Neumann boundary conditions at the top and bottom boundaries (see Remark~\ref{rem:def_tilde_bar_G}).
 The only difference in the analysis will be in the formulation of the boundary conditions in a version of problem \eqref{bar_V}
 for $\bar V^N:=\pt_\eta \bar v^N$, where $\bar v^N:= G^N-\bar G^N$, while $\bar G^N=\bar G^N_{{\mbox{\tiny$\Box$}}}$ is a version of \eqref{eq:def_tilde_bar_G} described in Remark~\ref{rem:def_tilde_bar_G}.
 Indeed, in this case $\pt_\eta \bar v^N=0$ also at the top and bottom boundaries,
 so a version of \eqref{bar_V_bc} for $\bar V^N$ becomes
      $\bar V(x,y;\xi,\eta)=0$ for $(\xi,\eta)\in\Omega$.
      This modification implies that a version of \eqref{bar_V_representation} for $\bar V^N$ will involve $G$ instead of $G^N$.

   Now, combining a version of \eqref{eq:thm:G_xi} for $\pt_\xi G^N$ with \eqref{new_barv_mixed},
   we arrive at $\|\pt^2_{\xi\eta}\bar v(x,y;\cdot)\|_{1\,;\Omega}
      \le C(\eps^{-1/2}+\mathcal{W}) (1+|\ln \eps|)$, where $\mathcal{W}\lesssim C$ (see part (ii)).
   Combining this with $\bar v= G-\bar G$ and the bound \eqref{bar_G_xi_xi_2} on $\pt^2_{\xi\eta}\bar G$
  yields the desired bound \eqref{eq:thm:G_xieta} on $\pt^2_{\xi\eta}G$.
  }
   \smallskip

   {\color{blue}(vi)} We now focus on the remaining assertion~\eqref{eq:thm:G_grad}.
   Rewrite the problem \eqref{prob_bar_v} as
   \[
      [-\eps(\partial^2_\xi+\partial^2_\eta)+1]\,\bar v(x,y;\xi,\eta)
      ={\bar h}_0(x,y;\xi,\eta)\quad\mbox{for}\;(\xi,\eta)\in\Omega,
      \qquad \bar v(x,y;\xi,\eta)\bigr|_{\pt\Omega}=0,
   \]
   where
   \begin{gather}\label{bar_bar_h_new}
      {\bar h}_0(x,y;\xi,\eta)
       :=  \bar h(x,y;\xi,\eta)
          -a(\xi,\eta)\,\pt_\xi \bar v(x,y;\xi,\eta)
          +[1+b(\xi,\eta)]\,\bar v(x,y;\xi,\eta).
   \end{gather}
   We shall represent $\bar v$ via the Green's function  $\Psi=\Psi(s,t;\xi,\eta)$
   of the two-dimensional self-adjoint operator
   $[-\eps(\partial^2_\xi+\partial^2_\eta)+1]$.
   Note that $\Psi$, for any fixed $(s,t)$, satisfies the equation
   $[-\eps(\partial^2_\xi+\partial^2_\eta)+1]\Psi(s,t;\xi,\eta)=\delta(\xi-s)\delta(\eta-t)$,
   and also the boundary conditions $\Psi(s,t;\xi,\eta)\bigr|_{(\xi,\eta)\in\pt\Omega}=0$.
   Furthermore, for any ball $B(x',y';\rho)$ of radius $\rho$ centred at any
   $(x',y')$, we cite the estimate \cite[(3.5b)]{Kopt08}
   \begin{gather}\label{G_2d}
      |\Psi(s,t;\cdot)|_{1,1\,;B(x',y';\rho)\cap\Omega} \le C\eps^{-1}\rho.
   \end{gather}
   The solution representation for $\bar v$ via $\Psi$ is given by
   \[
      \bar v(x,y;\xi,\eta)
        =\iint_\Omega \Psi(s,t;\xi,\eta)\,{\bar h}_0(x,y;s,t)\,ds\,dt.
   \]
   Applying $\pt_\xi$ and $\pt_\eta$ to this representation yields
   \begin{gather}\label{ber_v_new}
      |\bar v(x,y;\cdot)|_{1,1\,;B(x',y';\rho)\cap\Omega}
      \le \Bigl(\max_{(s,t)\in\Omega}|\Psi(s,t;\cdot)|_{1,1\,;B(x',y';\rho)\cap\Omega}\Bigr)\cdot
          \|{\bar h}_0(x,y;\cdot)\|_{1\,;\Omega}.
   \end{gather}
   To estimate $\|{\bar h}_0\|_{1\,;\Omega}$, recall that
   it was shown in part (iv) of this proof that
   $\| \pt_\xi \bar v\|_{1\,;\Omega} \le 2 \alpha^{-1} \|\bar{\bar h}\|_{1\,;\Omega}$
   and
   $\|\bar{\bar h}(x,y;\cdot)\|_{1\,;\Omega}\le C+\eps\mathcal{V}$,
   and in part (ii) that $\mathcal{V}\le C\eps^{-1}$.
   Consequently $\| \pt_\xi \bar v\|_{1\,;\Omega} \le C$.
   Combining this with \eqref{bar_bar_h_new} and \eqref{two_bounds}
   yields $\|{\bar h}_0\|_{1\,;\Omega}\le C$.
   In view of \eqref{ber_v_new} and \eqref{G_2d},
   we now get $|\bar v|_{1,1\,;B(x',y';\rho)\cap\Omega}\le C\eps^{-1}\rho$,
   which, combined with \eqref{bar_G_1_1_B},
   immediately gives the final desired bound \eqref{eq:thm:G_grad}.
 \end{proof}

\section{Generalisations}\label{sec:general}
\newcommand{\ve}[1]{\underline{#1}}
To generalise our results to more than two dimensions,
one needs to employ an $n$-dimensional version of the fundamental solution
$g$ of \eqref{eq:def_g0}, that will be denoted by $g_n$.
Let $x=(x_1,x_2,\ldots,x_n)$ and $(\xi_1,\xi_2,\ldots,\xi_n)$ be in $\R^n$,
and consider an $n$-dimensional version of problem \eqref{eq:Lu}
posed in the box domain $\Omega=(0,1)^n$,
with an $x_1$-direction of convection.
The corresponding constant-coefficient operator is
$-\eps \triangle_x-(2q)\, \pt_{x_1}$
(compare with the two-dimensional operator $\tilde L_{xy}$ of~\eqref{eq:Green_prim_const}),
where $\triangle_x:=\sum_{i=1}^n \pt^2_{x_i}$ is the standard $n$-dimensional
Laplacian.
For this operator a calculation yields the fundamental solutions
\[
   g_3(x,\xi)=\frac{1}{4\pi\eps}\,r^{-1}\,{e^{q(\xi_1-x_1-r)/\eps}},
   \qquad
   g_n(x,\xi)=\frac{1}{(2\pi\eps)^{n/2}}\,e^{q(\xi_1-x_1)/\eps}K_{n/2-1}(qr/\eps),
\]
where $r=|x-\xi|$,
and $K_{n/2-1}$ is
the modified Bessel function of second kind of (half-integer) order $n/2-1$.


   \bibliographystyle{plain}
   \bibliography{cd2d_mixed}

\begin{thebibliography}{10}

\bibitem{AS64}
M.~Abramowitz and I.~A. Stegun.
\newblock {\em {Handbook of Mathematical Functions with Formulas, Graphs and
  Mathematical Tables}}.
\newblock Applied Mathematics Series. National Bureau of Standards,
  Washingtion, D.C., 1964.

\bibitem{And_MM02}
V.~B. Andreev.
\newblock {A priori estimates for solutions of singularly perturbed two-point
  boundary value problems}.
\newblock {\em Mat. Model.}, 14(5):5--16, 2002.
\newblock {(in Russian)}.

\bibitem{And_2003}
V.~B. Andreev.
\newblock {Anisotropic estimates for the Green function of a singularly
  perturbed two-dimensional monotone difference convection-diffusion operator
  and their applications}.
\newblock {\em Comput. Math. Math. Phys.}, 43(4):521--528, 2003.
\newblock {Translated from Zh. Vychisl. Mat.Mat. Fiz., Vol 43, No. 4, 2003, pp.
  546--553}.

\bibitem{CK09}
N.M. Chadha and N.~Kopteva.
\newblock {Maximum norm a posteriori error estimate for a 3d singularly
  perturbed semilinear reaction-diffusion problem}.
\newblock {doi: 10.1007/s10444-010-9163-2 (published online 1 June 2010)},
  2010.

\bibitem{Dorf99}
W.~D\"orfler.
\newblock {Uniform a priori estimates for singularly perturbed elliptic
  equations in multidimensions}.
\newblock {\em SIAM J. Numer. Anal.}, 36(6):1878--1900, 1999.

\bibitem{erikss}
K.~Eriksson.
\newblock {An adaptive finite element method with efficient maximum norm error
  control for elliptic problems}.
\newblock {\em Math. Models Methods Appl. Sci.}, 4:313--329, 1994.

\bibitem{FK10_NA}
S.~Franz and N.~Kopteva.
\newblock {A posteriori error estimation for a convection-diffusion problem
  with characteristic layers}.
\newblock in preparation, 2010.

\bibitem{FK09_lower}
S.~Franz and N.~Kopteva.
\newblock {On the sharpness of Green's function estimates for a
  convection-diffusion problem}.
\newblock In M. Koleva and L. Vulkov, eds. {\it Finite Difference Methods:
  Theory and Applications, Proceedings of the Fifth International Conference
  FDM: T\&A'10, Lozenetz, Bulgaria, June 2010.} Rousse University Press,
  Rousse, 2011.

\bibitem{FK10_1}
S.~Franz and N.~Kopteva.
\newblock {Green's function estimates for a singularly perturbed
  convection-diffusion problem}.
\newblock {\em J. Differential Equations}, 252:1521--1545, 2012.

\bibitem{Griffel}
D.~H. Griffel.
\newblock {\em {Applied functional analysis}}.
\newblock Dover Publications, Mineola, NY, 2002.

\bibitem{Leyk}
J.~Guzm\'{a}n, D.~Leykekhman, J.~Rossmann, and A.~H. Schatz.
\newblock {H\"older estimates for Green's functions on convex polyhedral
  domains and their applications to finite element methods}.
\newblock {\em Numer. Math.}, 112:221--243, 2009.

\bibitem{Ilin}
A.~M. Il'in.
\newblock {\em {Matching of asymptotic expansions of solutions of boundary
  value problems}}, volume 102 of {\em Translations of Mathematical
  Monographs}.
\newblock American Mathematical Society, Providence, RI, 1992.

\bibitem{KSh87}
R.~B. Kellogg and S.~Shih.
\newblock {Asymptotic analysis of a singular perturbation problem}.
\newblock {\em SIAM Journal on Mathematical Analysis}, 18(5):1467--1510, 1987.

\bibitem{KSt05}
R.~B. Kellogg and M.~Stynes.
\newblock {Corner singularities and boundary layers in a simple
  convection-diffusion problem}.
\newblock {\em J. Differential Equations}, 213(1):81--120, 2005.

\bibitem{KSt07}
R.~B. Kellogg and M.~Stynes.
\newblock {Sharpened bounds for corner singularities and boundary layers in a
  simple convection-diffusion problem}.
\newblock {\em Appl. Math. Lett.}, 20(5):539--544, 2007.

\bibitem{Kopt08}
N.~Kopteva.
\newblock {Maximum norm a posteriori error estimate for a 2d singularly
  perturbed reaction-diffusion problem}.
\newblock {\em SIAM J. Numer. Anal.}, 46:1602--1618, 2008.

\bibitem{Lady68}
O.~A. Ladyzhenskaya and N.~N. Ural'tseva.
\newblock {\em {Linear and Quasilinear Elliptic Equations}}.
\newblock Academic Press, New York, 1968.

\bibitem{Linss}
T.~Lin\ss.
\newblock {\em {Layer-adapted meshes for reaction-convection-diffusion
  problems}}.
\newblock Lecture Notes in Mathematics. Springer-Verlag, Berlin, 2010.

\bibitem{notch}
R.~H. Nochetto.
\newblock {Pointwise a posteriori error estimates for elliptic problems on
  highly graded meshes}.
\newblock {\em Math. Comp.}, 64:1--22, 1995.

\bibitem{RST08}
H.-G. Roos, M.~Stynes, and L.~Tobiska.
\newblock {\em {Robust numerical methods for singularly perturbed differential
  equations}}, volume~24 of {\em Springer Series in Computational Mathematics}.
\newblock Springer-Verlag, Berlin, second edition, 2008.

\bibitem{Shi92}
G.~I. Shishkin.
\newblock {\em {Grid Approximations of Singularly Perturbed Elliptic and
  Parabolic Equations}}.
\newblock Russian Academy of Sciences, Ural Section, Ekaterinburg, 1992.
\newblock (in Russian).

\bibitem{Stakgold2}
I.~Stakgold.
\newblock {\em Boundary value problems of mathematical physics (vol. 2)}.
\newblock Society for Industrial and Applied Mathematics, Philadelphia, PA,
  USA, 2000.

\end{thebibliography}
\end{document}